\documentclass{article}
\usepackage{latexsym,amssymb,amsfonts}
\usepackage{amsmath}
\usepackage{amsthm}
\usepackage{amscd}
\usepackage{graphicx}
\usepackage[arc,all]{xy}
\usepackage{tikz}
\usepackage{braids}
\usepackage{hyperref}
\usepackage{pict2e}
\usepackage{ifpdf}

\theoremstyle{definition}
\theoremstyle{example}
\theoremstyle{remark}
\newtheorem{theorem}{Theorem}[section]
\newtheorem{proposition}{Proposition}[section]
\newtheorem{lemma}{Lemma}[section]
\newtheorem{definition}{Definition}[section]
\newtheorem{example}{Example}[section]
\newtheorem{remark}{Remark}[section]
\newtheorem{corollary}{Corollary}[section]

\begin{document}

\markboth{Uwe Kaiser}
{Homflypt skein modules, string topology and $2$-categories}

\title{HOMFLYPT SKEIN THEORY, STRING TOPOLOGY AND $2$-CATEGORIES}

\maketitle

\centerline{\author{UWE KAISER}}

\centerline{\textit{Department of Mathematics}}
\centerline{\textit{Boise State University, 1910 University Drive}}
\centerline{\textit{ Boise, ID 83725-1555, ukaiser@boisestate.edu
}}

\begin{abstract}
We show that relations in
Homflypt type skein theory of an oriented $3$-manifold $M$ are induced
from a $2$-groupoid defined from the
fundamental $2$-groupoid 
of a space of singular links in $M$.
The module relations are  
defined by homomorphisms related to string topology. They appear from a
\textit{representation} of the groupoid into free modules on a set of model
objects.
The construction on the fundamental $2$-groupoid
is defined by the singularity stratification and relates Vassiliev and skein theory.  
Several explicit properties are discussed, and some implications for 
skein modules are derived.
\end{abstract}

\centerline{\textit{Keywords:}{Skein Modules, String topology, 2-categories}}

\centerline{Mathematics Subject Classification 2010: 57M25, 57M35, 57R42}

\section{Introduction and summary of results}

The first version of this article was written in 2003, following results by the author on
homotopy skein modules \cite{K1}, \cite{K2} and two-term skein modules of framed links\cite{K3}. The techniques used in this paper are extending the ideas of those papers and modifications of the original ideas of Lin \cite{L}, Kalfagianni \cite{Ka1}, Kalfagianni-Lin \cite{KaL} and Kirk-Livingston \cite{KiL}, who construct finite type invariants of links in $3$-manifolds. They observed essential singular spheres and tori as the obstructions in the construction of finite type invariants. The intention of this article is to apply mapping space techniques in a more directl way to the calculation of skein modules. In this way, without passing to \textit{dual} modules, we will be able to describe the possible torsion in the skein modules.

\vskip 0.1in

The problem to find presentations of skein modules naturally splits into \textbf{(1)} to determine a \textit{minimal} set of generators, and \textbf{(2)} to find relations in terms of these relations. In 2001, Charlie Frohman pointed out to the author that the relations described in \cite{K1} are related to the string topology operations just introduced by Chas-Sullivan \cite{CS1}. While the ideas from string topology give a pretty good way to approach \textbf{(2)} the problem \textbf{(1)} still remains evasive: It turns out to be difficult to determine whether the obvious minimal set of generators is sufficient (It is a challenging conjecture that this is even true for hyperbolic $3$-manifolds). It has been a lack of explicit applications, which put a hold on the project described here at that time. But the author kept the hope that there would be substantial progress with problem \textbf{(1)}, thereby providing stronger impact to the methods developed in this article. 
The author's decision to revise the original manuscript for publication now is partly based on recent progress in this direction: in fact problem \textbf{(1)} for $L(p,q)$ was solved by \cite{C}, and later on by \cite{GM}, even though in both cases a careful analysis of their arguments is necessary (and will not be given in this article). The result that the Homflypt skein modules of lens spaces $L(p,q)$ (except for $S^2\times S^1$) are free on the expected set of generators for $q>1$ is still not published until now. The two main approaches to find presentations of skein modules of $L(p,q)$ with respect to problem \textbf{(2)} are using projections \cite{GM} respectively braids \cite{DLP}, \cite{DL1}, \cite{DL2}. This work is still in progress but until now completely successful only for $q=1$. A second reason to revisit the manuscript from 2003 came from increasing interest in higher category theory in low dimensional topology during recent years, motivation from categorification but also new relations to the topology of stratified spaces \cite{T}. Also note that, in contrast to our skein modules, the groundbreaking recent result of Gunningham, Jordan and Safronov \cite{GJS}, Ddefinition 2.2. and Theorem 4.8, stating that the skein modules $\textrm{Sk}_G(M)$ are finite-dimensional $\mathbb{C}[q^{1/d}]$-vector spaces for each reductive group $G$ and closed, oriented $3$-manifold $M$. 

\vskip 0.1in

In this article we will circumvent \textbf{(1)} and replace skein modules by adic completions, which in the cases where \textbf{(2)} gives only trivial relations actually turns out to be equivalent to the construction of power series invariants previously constructed from finite type invariants in \cite{L}, \cite{Ka1}, \cite{Ka2}, \cite{KaL}.

\vskip .1in

On a theoretical level we develop ideas using category theory
to study the stratified topology of mapping spaces associated with skein modules. We will see
how the structures of skein modules and string topology naturally 
combine on a categorical level. This gives
both a theoretical
understanding and new tools for the computation of skein modules of
links in oriented $3$-manifolds. 

\vskip .1in

Recall that a skein module of an oriented $3$-manifold $M$ is the
quotient of a free module with basis a set of  
links (possibly including
singular links) in $M$ by a submodule generated by linear
combinations of links defined from local tangle modifications in
oriented $3$-balls.
In quantum field theory the
skein modules of $M$, or rather their duals, appear as modules of
quantum observables.
In link theory the skein module is the target of a 
\textit{universal link invariant} satisfying  
given skein relations. Skein \textit{algebras} have been 
studied in detail for cylinders over oriented surfaces, But not much is known 
about the structure of skein modules for general $3$-manifolds. 
For an overview
see \cite{P1}.

\vskip .1in

We will construct various types of link invariants 
satisfying skein relations. Our construction is algebraic
and follows from the structure of a $2$-groupoid defined
in section 3 using the general ideas from section 2.

\vskip .1in 

Throughout let $M$ be a compact connected oriented $3$-manifold.
For $k\geq 0$ let $\mathcal{L}[k]$ be the set of 
$k$-links in $M$.
The elements of $\mathcal{L}[k]$ are the isotopy class of
$k$-embeddings, i.\ e.\ immersions of circles in $M$ with singularity 
given by precisely $k$ double-points without tangencies.
Thus $\mathcal{L}[0]$ is the usual set of oriented links in $M$
(including the empty link).
For $K_*\in \mathcal{L}[1]$ let $K_{\pm}, K_0$ be the \textit{Conway resolutions}.
Similarly let $\mathcal{L}_{\mathfrak{f}}[k]$ denote the set of isotopy classes of framed 
oriented links with precisely $k$ double-points without tangencies. 
If a fixed $3$-manifold is clear from the context we often just use notation $\mathcal{L}$.

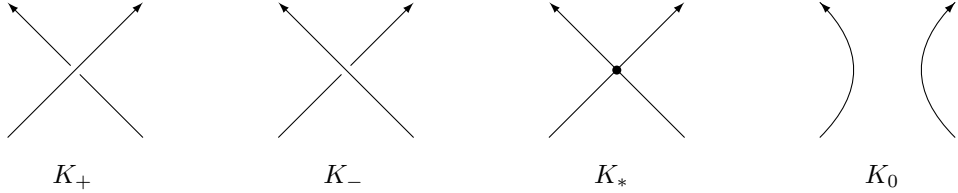
\setlength{\unitlength}{0.6cm}
\begin{picture}(5,5)
\put(-1,0){\vector(1,1){3}}
\put(2,0){\line(-1,1){1.4}}
\put(0.4,1.6){\vector(-1,1){1.4}}
\put(0,-1){$K_+$}
\put(5,0){\line(1,1){1.4}}
\put(6.6,1.6){\vector(1,1){1.4}}
\put(8,0){\vector(-1,1){3}}
\put(6,-1){$K_-$}
\put(11,0){\vector(1,1){3}}
\put(14,0){\vector(-1,1){3}}
\put(12.5,1.5){\circle*{0.2}}
\put(12,-1){$K_*$}
\qbezier(17,0)(18.5,1.5)(17,3)
\qbezier(20,0)(18.5,1.5)(20,3)
\put(17.1,2.9){\vector(-1,1){0.1}}
\put(19.9,2.9){\vector(1,1){0.1}}
\put(18,-1){$K_0$}
\end{picture}

\vskip 1.0cm

Let $\mathcal{R}$ be a commutative ring with $1$ and let $\mathcal{R}^{\times }$ be the corresponding set 
of units.
Let $q,v\in \mathcal{R}^{\times }$ (which may very well be $1$). Let $I\subset J$ be proper ideals of $\mathcal{R}$.
For each $\mathcal{R}$-module $W$ let $W[[I]]$ denote the 
$I$-adic completion of $W$, i.\ e.\ the inverse limit of the
system $W/I^iW$ for $i\geq 0$.   

A map $\sigma : \mathcal{L}[1]\rightarrow \mathcal{R}\mathcal{L}[0]$
is called a \textit{skein potential} (with respect to $I$) 
if the image of a $j$-component 
link is a linear combination of links with $<j$ components
and coefficients in $J\setminus I$,
and links with $\geq j$ components and
coefficients in $I$. 

We want to study skein relations
of oriented links in $M$ of the form
$q^{-1}K_+-qK_-=\sigma (K_*)$,
and skein relations for framed oriented
links: $q^{-1}K_+-qK_-=\sigma (K_*)$, $K^{(+)}=qv^{-1}K$,
where $K^{(+)}$ is the result of introducing a positive twist into a
component of $K$.
  
\vskip .1in

In Chern Simons theory the 
\textit{classical} observables 
are the homotopy
classes of oriented links in $M$ (holonomy along loops).
Thus it is natural to consider skein modules 
as \textit{deformations} of free objects generated by 
homotopy classes of links in $M$.
This is our starting point.

\vskip .1in

Let $\mathfrak{b}[0]$
denote the set of homotopy classes of oriented links in $M$. 
We will choose \textit{geometric models}, also called 
\textit{standard links}, i.\ e.\ a map $\mathfrak{b}[0]\rightarrow
\mathcal{L}[0]$ assigning to each homotopy class of link  
\textit{some} oriented link with the
given homotopy class. 
Let
$$\partial: \mathcal{L}[k]\rightarrow \mathcal{L}[k-1],$$
defined by $\partial (K_*)=q^{-2}K_+-K_-$.
($*$ denotes the last double-point),  
be the \textit{Jones boundary}. For $q=1$ it is called the
\textit{Conway boundary}. 
Let
$$\mathcal{S}(\sigma ):=\mathcal{R}\mathcal{L}[0]/im(\partial
-q^{-1}\sigma)$$
be the \textit{skein module} defined by $\sigma $. 

In section 2 we will define the notion of \textit{insensitive} 
skein potential.
The usual oriented or framed oriented \textit{local} skein relations 
are defined by insensitive
skein potentials. 

The idea of the skein module is a \textit{finite formal expansion} of links 
using the skein relation repeatedly. We will study expansions of links
in terms of standard links by possibly infinite applications of the skein relation.
This is formalized in the following result.

\begin{theorem} Let $\sigma $ be a skein
potential and $\mathfrak{s}$ be a 
choice of geometric models.
Then there is  
defined a submodule 
$$U=U(\sigma, \mathfrak{s})\subset \mathcal{R}(\mathfrak{b}[0])[[I]],$$ 
snd the $I$-adic link invariant
$$\rho =\rho(\sigma, \mathfrak{s}): 
\mathcal{L}[0]\rightarrow \mathcal{R}(\mathfrak{b}[0])[[I]]/U.$$
The submodule $U$ is generated by \textup{(i)} expansion of 
differentiability relations (see section \textup{4}), and \textup{(ii)} 
expansion of elements in the image of a transversal 
string topology homomorphism (defined by $\sigma $). 
This homomorphism is
defined on $1$-dimensional homology
groups of the space of immersions in $M$. 
If $\sigma $ is insensitive then $U$ is determined only by \textup{(ii)}.
The map $\rho$ induces the module isomorphism
$$S(\mathfrak{\sigma })[[I]]\cong \mathcal{R}\mathfrak{b}[[I]]/U.$$
The composition
$$\rho (\sigma, \mathfrak{s})\circ \mathfrak{s}:
\mathfrak{b}[0]\rightarrow \mathcal{R}(\mathfrak{b}[0])[[I]]/U$$
is the composition of the natural map
$\mathfrak{b}[0]\rightarrow \mathcal{R}(\mathfrak{b}[0])[[I]]$ 
with the natural
projection.
\end{theorem}

We will show that for insensitive $\sigma $ 
the submodule $U$ is completely 
determined by loops in the space of immersions defined from kink crossings,
and by essential singular tori in $M$. This corresponds to previous results of
Kalfagianni \cite{Ka1}, \cite{Ka2} and Kalfagianni and Lin \cite{KaL}.

\vskip .1in

The module $\mathcal{R}(\mathfrak{b}[0])[[I]]$ is an algebra with multiplication 
defined by formal multiplication of homotopy classes of maps.
This multiplication induces on $S(\sigma )[[I]]$ 
the structure of a \textit{skein algebra} if $U$ is an
ideal in $\mathcal{R}(\mathfrak{b}[0])$. The precise geometric relevance 
of such an algebra structure seems not yet have been studied.  

\vskip 0.1in

\noindent \textit{Question.}\
For which $3$-manifolds $M$ and skein relations $\sigma $ 
can we choose $\mathfrak{s}$ such that
$U(\sigma ,\mathfrak{s})$ is an ideal? How can we modify the ring
$\mathcal{R}$ to 
define a multiplication on $\mathcal{S}(\sigma )$, which \textit{deforms}
the multiplication of free homotopy classes in $M$. (This is known to 
be true for cylinders over surfaces, \cite{P2}, \cite{Tu}.)

\vskip 0.1in

From Theorem 1.1 some explicit results
concerning classical skein modules can be deduced.
$\mathcal{R}:=\mathbb{Z}[q^{\pm 1},z,h]$. 
Then $z\neq h\in \mathcal{R}$ are non-invertible elements. Let
$\sigma (K_*)=hK_0$ for a self-crossing and $\sigma (K_*)=zK_0$ 
for a crossing of different components of $K_*$. The skein module 
$\mathcal{J}(M)$ is defined by the free $\mathcal{R}$-module with 
basis isotopy classes of oriented links, 
and relations 
$q^{-1}K_+-qK_-=\sigma (K_*)$ is the called the \textit{generalized 
Jones module} and has been 
considered in \cite{Tu}, \cite{P3}. The skein module 
$\mathcal{H}(M)$ is defined by framed oriented links with
$\mathcal{R}:=\mathbb{Z}[q^{\pm 1},v^{\pm 1},z,h]$ and relations 
$q^{-1}K_+-qK_-=\sigma (K_*), K^{(+)}=qv^{-1}K$. This
is the \textit{(variant)} Homflypt \textit{skein module} (with split
variables). For $z=h=s^{-1}-s$ this module
has been discussed e.\ g.\ in \cite{M}, see also \cite{GZ1},
\cite{GZ2}
for results about
of $S^2\times S^1$ and connected sums.

The skein modules defined by imposing additional vaccuum relations  
$(q^{-1}-q)\emptyset=hU$ respectively $(v-v^{-1})\emptyset=hU$ for the
unknot $U$ and the empty link $\emptyset$ will be denoted
$\mathfrak{J}(M)$ and $\mathfrak{H}(M)$ respectively. 

The skein modules of oriented links $\mathfrak{J}(M)$ respectively framed links
$\mathfrak{H}(M)$ are modules over the corresponding skein modules 
of the $3$-ball, with module actions defined by disjoint union with links in a separated $3$-ball
(here is where we use that $M$ is connected).
Let $\mathfrak{R}$ be the
corresponding skein module of the $3$-ball. 
It is known (\cite{P2} and \cite{Tu}) that
$$\mathfrak{R}\cong \mathbb{Z}[q^{\pm 1},z,h,\frac{q^{-1}-q}{h}]$$
and
$$\mathfrak{R}\cong \mathbb{Z}[q^{\pm},v^{\pm},z,h,\frac{v-v^{-1}}{h}],$$
respectively.
Let $\mathfrak{b}_0$ be the set of homotopy classes of links in
$M$ without homotopically trivial components. Then there are defined
$\mathfrak{R}$-homomorphisms $\mathfrak{s}$ 
from the free modules with basis $\mathfrak{b}_0$ into
the modules $\mathfrak{J}(M)$ respectively $\mathfrak{H}(M)$. 
Then we have
$$\mathfrak{R}\mathfrak{b}_0\cong S\mathfrak{R}\hat{\pi}^0,$$
where $\hat{\pi}^0$ is the set of non-trivial conjugacy classes of the
fundamental group of $M$.
Here the \textit{geometric model} $\mathfrak{s}$  
assigns to a monomial in 
$\hat{\pi}^0$ the oriented link with the homotopy classes of
components given
by the monomial. Note that there is also the trivial monomial $1\in \mathfrak{b}_0$.
The geometric model is called \textit{nice} 
if the following holds: If some element of
$\hat{\pi }^0$ appears repeatedly in the sequence we assume that there
exist self-isotopies of the representing link, which arbitrarily change the
order of components with the same free homotopy classes.
Moreover we assume that multiplication by $\frac{q^{-1}-q}{h}$
respectively $\frac{v-v^{-1}}{h}$ corresponds to adding some unknotted
unlinked component to the corresponding standard link.

A $3$-manifold $M$ is called \textit{atoroidal} respectively 
\textit{aspherical} if each 
essential (i.\ e.\ the induced homomorphism of fundamental groups is 
injective) map of a torus
$S^2\times S^1$ respectively map of a $2$-sphere $S^2$ in $M$ is
homotopic into the boundary of $M$. 

\begin{theorem} Suppose that $M$ is atoroidal and aspherical. Then the
submodules of $\mathfrak{J}(M)$ respectively $\mathfrak{H}(M)$, which
are generated
by the image of a nice geometric model $\mathfrak{s}$, are isomorphic to
$S\mathfrak{R}\hat{\pi}^0$.
\end{theorem}

In section 9 we will see that 
the following result can be used to reproduce some of the results of Kalfagianni \cite{Ka1}.

\begin{theorem} Suppose that $M$ is atoroidal and aspherical. 
Then for each $\beta \in \mathfrak{b}_0$ there is defined a $\mathfrak{R}$-homomorphisms $\iota_{\beta }$
from $\mathfrak{J}(M)$ respectively $\mathfrak{H}(M)$ into 
$\mathfrak{R}[[h]]$ satisfying
$\iota_{\beta }\circ \mathfrak{s}(\beta ')=\delta_{\beta ,\beta '}$
for all $\beta ,\beta ' \in \mathfrak{b}_0$.
\end{theorem}

If a geometric model map is onto and there are no relations then a basis is given by
the geometric models. Using that the arguments contained in the proofs 
in \cite{GM} actually show that there are nice geometric models generating the skein module, 
we deduce the following result. 

\begin{theorem}
Let $M=L(p,q)$ be a Lens space and $p\neq 0$. Then $\mathfrak{J}(M)$
respectively $\mathfrak{H}(M)$ is 
isomorphic to $S\mathfrak{R}\hat{\pi}^0$.
\end{theorem}

It follows now easily from Przytycki's universal 
coefficient theorem
for skein modules that the usual variant Homflypt module with
$z=h=:s-s^{-1}$
is a free module over $\mathbb{Z}[q^{\pm 1},v^{\pm 1},
\frac{v-v^{-1}}{s-s^{-1}}]$, and similarly the result for the 
usual Homflypt skein module defined over $\mathbb{Z}[q^{\pm 1},h^{\pm 1}]$. 

\vskip .1in

The combined results of sections 2 and 3 prove the following result, similarly to Theorem 1.1 for the \textit{Vassiliev relation} as described below. 

\vskip .1in

Let $\mathcal{L}:=\cup_{k\geq 0} \mathcal{L}[k]$ be the set of all singular
links in $M$.

Let $\mathcal{R}:=\mathbb{Z}[q^{\pm 1},h]$.
We define the \textit{Vassiliev potential} (not a skein potential according to the definition above) by
$\sigma_V : \mathcal{L}\rightarrow I\mathcal{L}\subset \mathcal{R}\mathcal{L}$ by 
$\sigma _V(K_*)=hK_*$, where $I$ is the ideal generated by $h$.
Then we can define the \textit{infinite Vassiliev module}
$$\mathcal{S}(\sigma_V):=\mathcal{R}\mathcal{L}/im(\partial-q^{-1}\sigma).$$

\vskip 0.1in

Let $\mathfrak{b}[k]$ denote the set of homotopy classes of singular
$k$-links, see \cite{A} and \cite{CR}. Then $\mathfrak{b}$ is in one-to-one correspondence with 
a certain set of chord diagrams in $M$ up to homotopy, with specified ordering of the chords.
We choose geometric models $\mathfrak{s}$, i.\
e.\ singular links realizing the given homotopy classes.
Let $\mathfrak{b}:=\cup_{k\geq 0}\mathfrak{b}[k]$.

\begin{theorem}
Let $\mathfrak{s}$ be a given choice of geometric models.
There is defined a submodule $U_V\subset \mathcal{R}\mathfrak{b}[[I]]$ and map
$$\rho (\sigma_V,\mathfrak{s}): \mathcal{L}\rightarrow
\mathcal{R}\mathfrak{b}[[h]]/U_V,$$
which induces the isomorphism
$$\mathcal{S}(\sigma_V)[[h]]\rightarrow \mathcal{R}\mathfrak{b}[[h]]/U_V.$$
The submodule $U_V$ is generated by \textup{(i)} expansions of    
differentiability relations, $4T$-relations and tangency relations
(defined from the local
stratification of the space of immersions), and \textup{(ii)} expansions of
images of generalized string topology homomorphisms defined on 
the $1$-dimensional homology of spaces of $k$-immersions.
The composition
$$\rho (\sigma_V,\mathfrak{s})\circ \mathfrak{s}:
\mathcal{R}\mathfrak{b}\rightarrow
\mathcal{R}\mathfrak{b}[[h]]/U_V$$
is the composition of the natural map $\mathfrak{b}\rightarrow
\mathcal{R}\mathfrak{b}[[h]]$ with the natural projection.
\end{theorem}

The expansion map $\rho (\sigma_V,\mathfrak{s})$ 
in Theorem 1.5 is defined 
on \textit{all} singular links.
The \textit{local} differentiability and tangency
relations can be subsumed 
in terms of the ordering
of chords. But the geometric $4T$-relations are very hard to control
explicitly.  
This has been the obstruction for a direct topological construction of
Kontsevitch's integral from mapping space topology in the early $90's$  \cite{BS}, \cite{H} which is still open. 

\vskip .1in

\noindent \textit{Question.}\
For which $3$-manifolds $M$ is the homomorphism
$$\mathcal{S}(\sigma_V)\rightarrow \mathcal{S}(\sigma_V)[[h]]$$
injective?
The map $\rho (\sigma_V,\mathfrak{s})$ appears as a kind of limit of Vassiliev
type invariants and thus could be related to questions about the strength of Vassiliev imvariants.

\vskip 0.1in

The relation between skein theory and Vassiliev theory has
been observed by Przytyki, see \cite{P4},  
\cite{Kas} and \cite{BL}.
For $i \geq 0$ let $\mathcal{G}_i(M;\mathcal{R})$ be the quotient
of $\mathcal{R}\mathcal{L}$ by the submodule generated by relations 
$K_*=K_+-K_-$ for \textit{all}
$K_*\in \mathcal{L}$ and $*$ \textit{any} double-point, 
and $K=0$ for each singular
link $K$ with $>i$ double-points.  
Note that $\mathcal{G}_i(M;\mathcal{R})$ is generated by 
either $\mathcal{L}[0]$ or by
the set of homotopy classes of singular $k$-links with $k \leq i$.
The dual of $\mathcal{G}_i(M;\mathcal{R})$ is the module of
$\mathcal{R}$-valued Vassiliev
invariants of order $\leq i$ for links in $M$.
In particular $\mathcal{G}_0(M;\mathcal{R})$ is dual to the module of type
$0$-invariants. This module is isomorphic to $S\mathcal{R}\hat{\pi }$, where
$\hat{\pi }$ is the set of conjugacy classes of the fundamental
group of $M$. The module $\mathcal{G}_1(M;\mathcal{R})$ is the dual of the 
module of type
$1$-invariants, and has been discussed in \cite{KiL}.
It has been observed in \cite{CR} that a generalized 
string topology homomorphism
provides the indeterminancy of the universal type $1$-invariant.
The observation for homotopy skein modules has been made
by the author in \cite{K1}. 

Let
$$\mathcal{G}(M;\mathcal{R}):=\lim_{\leftarrow }
\mathcal{G}_i(M;\mathcal{R})$$
be the \textit{finite Vassiliev module} of $M$. It has been observed by Przytycki that
each skein relation defined by a skein potential $\sigma $ 
(for $\partial (K_*)=K_+-K_-$) defines an epimorphism
$$\mathcal{G}(M;\mathcal{R})\rightarrow \mathcal{S}(\sigma )[[I]],$$
by mapping $K_*\in \mathcal{L}[1]$ to $K_+-K_- -\sigma(K_*)$.

The relation between the infinite and finite Vassiliev module is subtle. The fact that 
$h$ is non-invertible is important here.  
Assume we are in the Conway case $q=1$.
For $i\geq 0$ let
$\mathcal{W}_i(M;\mathcal{R})$ be the quotient of
$\mathcal{R}\mathcal{L}$ by
relations $K_+-K_--K_*$ for $*$ the \textit{last}
double-point, and $K=0$ for singular links with $\geq i$
double-points. Applying the relation to $K_{**}$ we get
$K_{**}=K_{*+}-K_{*-}$. This is equal to $K_{+*}-K_{-*}$, which is the
skein relation applied to another double-point, if and only
if the differentiability relations hold.
Note that $\mathcal{S}(\sigma_V)/h^i\mathcal{S}(\sigma_V)$ 
is \textit{not} generated by
links because we can only use $hK_*=K_+-K_-$ to reduce the number
of double-points. Roughly, this results in the fact that 
the order of chords in chord diagrams representing homotopy classes of
singular links, can be changed only in corresponding powers of $h$.
 
\vskip .1in

Our Theorems 1.1 and 1.5 are linearized versions of a 
general construction on $2$-groupoids equpped with special structure. 
These $2$-groupoids are 
constructed from certain $3$-stage stratifications 
in spaces of $k$-immersions (immersions with $k$ fixed double points, see section 3 for the precise definitions).
We will construct a type of
\textit{deformation} 
of the usual homotopy $2$-groupoid of the 
space of $k$-immersions
of circles in $M$. It will be shown that the resulting
algebraic $2$-groupoid 
induces the usual skein and Vassiliev structures.
The isotopy classes of links
appear as the objects of a
category, transversal deformations between links are the $1$-morphisms
of the category.
It should be possible to generalize the 
$2$-groupoids constructed
in this article to 
$n$-groupoids, deforming
the usual homotopy $n$-groupoids 
of the spaces of $k$-immersions. It could be a future goal to construct 
deformation structures of string topology in this
setting, which could be a future project. In the language
of $2$-groupoids the \textit{string topology} homomorphisms
are the structure, which relate $2$-morphisms and equivalence 
classes of $1$-morphisms.

It will be important that the categorical 
structure above can be \textit{linearized} similarly to the quantum invariant
setting.
Our viewpoint in this paper is much more general than
necessary to deduce the explicit results above. But it is interesting to
realize that \textit{the structures inherent in skein modules and
Vassiliev invariants,  are a necessary consequence of the groupoid
structure}.

Most results of this article could be extended easily to suitable tangles in $3$-manifold following the definitions in
\cite{Ha}, section 6. We decided to keep with the case of links to keep the notation less difficult.   

\section{Formal skein theory}

For details about category and higher category theory see \cite{BD},
(\cite{Ke}, 0.2), \cite{La} and \cite{Ma}.

\vskip .1in

A \textit{$2$-category} is a triple 
$(\mathfrak{ob},\textrm{hom},\textrm{mor})=(\textrm{hom}_0,\textrm{hom}_1,\textrm{hom}_2)$,
for which both $(\mathfrak{ob},\textrm{hom})$ and $(\textrm{hom},\textrm{mor})$ are  
categories in the usual sense. More precisely, the following
structures are given. 
For $x,y\in \mathfrak{ob}$ there is defined 
$\textrm{hom}(x,y)$ with composition: 
$$ \circ : \textrm{hom}(x,y)\times \textrm{hom}(y,z)\rightarrow \textrm{hom}(x,z).$$
The elements in $\textrm{hom}(x,y)$ are called
$1$-morphisms. Composition is often abbreviated $u\circ v=:uv$.
The morphisms between $1 
$-morphisms are called $2$-morphisms.
The set of $2$-morphisms from $u$ to $v$ is denoted $\textrm{mor}(u,v)$,
and is non-empty only for \textit{parallel} objects $u,v\in \textrm{hom}(x,y)$.
There are two compositions of $2$-morphisms.  
For $u,v,w\in \textrm{hom}(x,y)$ there is defined:
$$\circ _1 : \textrm{mor}(u,v)\times \textrm{mor}(v,w)\rightarrow \textrm{mor}(u,w).$$
For $x,y,z\in \mathfrak{ob}$, $u_1,u_2\in \textrm{hom}(x,y)$ and $v_1,v_2\in \textrm{hom}(y,z)$
there is defined
$$\circ _2: \textrm{mor}(u_1,u_2)\times \textrm{mor}(v_1,v_2)\rightarrow \textrm{mor}(v_1\circ
u_1,v_2\circ u_2).$$
Recall the there are the natural source and target maps denoted by
$\textrm{sour}$ and $\textrm{targ}$, defined $\textrm{hom} \rightarrow \mathfrak{ob}$ and $\textrm{mor} \rightarrow \textrm{hom}$ respectively. 

A $2$-category in which \textit{all} morphisms are
equivalences is called a \textit{$2$-groupoid}. Equivalence for $1$-morphisms is in the weak sense \cite{La}.
Thus for $u\in \textrm{hom}(x,y)$ there is exists $v \in \textrm{hom}(y,x)$
such that both $u\circ v$ and $ v \circ u$ are equivalent to
the corresponding identity morphisms $\textrm{id}$. 
Equivalence means here that
$v \circ u$ and $\textrm{id}$ are related by a $2$-morphism, i.\
e.\
$\textrm{mor}(v \circ u,\textrm{id})\neq \emptyset$.
The $2$-morphisms have to be invertible \textit{on the nose}, i.\ e.\ for a $2$ morphism $h\in \textrm{mor}(u,v)$ there exists a $2$-morphism $h^{-1}\in \textrm{mor}(v,u)$ such that $h\circ_1 h^{-1}=\textrm{id}_u$ and 
$h^{-1}\circ_1 h=\textrm{id}_v$, where $\textrm{id}_u, \textrm{id}_v$ are the corresponding identity $2$-morphisms
in $\textrm{mor}(u,u)$ respectively $\textrm{mor}(v,v)$. 
A $2$-groupoid \textit{with inverses} is a small $2$-groupoid with 
a choice of $v:=u^{-1}\in \textrm{hom}(y,x)$ for each $u\in \textrm{hom}(x,y)$.
Recall that an $n$-category is small if the class of objects and
all classes of morphisms are sets.
In the following without mentioning \textit{we will assume that all our
$2$-groupoids are small and with inverses}.
  
The composition of $1$-morphisms is associative up to 
the action of $2$-morphisms. The set of $2$-morphisms $\textrm{mor}(u,u)$ 
is a group under $\circ _1$.
For $x \in \mathfrak{ob}$ the set of equivalence classes of elements in 
$\textrm{hom}(x,x)$ under the action of $\textrm{mor}$
is a group with composition defined by 
$[u]\cdot [v]:=[u\circ v]$. (This is well-defined because of the
composition $\circ _2$.)
Here $u,v\in \textrm{hom}(x,x)$ are equivalent if $\textrm{mor}(u,v)\neq \emptyset$.
 
\begin{definition} A \textit{$2$-category with models} is a  
$2$-categeory with a distinguished subset
$\mathfrak{m}\subset \mathfrak{ob}$ such
that for each $x\in \mathfrak{ob}$ there exists  
$b \in \mathfrak{m}$ such that $\textrm{hom}(x,b)\neq
\emptyset$. If $\textrm{hom}(b_1,b_2)=\emptyset$ for all
$b_1\neq b_2$ then $\mathfrak{m}$ is called
a \textit{minimal set of models} of the $2$-category.
\end{definition}

For each set $X$ let
$F(X)=:F^{(0)}(X)$ be the free group 
generated by the elements of $X$.
Let $F^{(j)}(X):=F(F^{j-1}(X))$ be defined
inductively for $j\geq 1$.
Let  
$$\hat{F}(X):=\prod_{j\geq 0}F^{(j)}(X)
$$
be the group defined by the infinite product.
The group $\hat{F}(X)$ is not a free group but inverse
limit of free groups.

\begin{example} The group $\hat{F}(X)$ is always
non-commutative for $X\neq \emptyset$. Let $X=\{* \}$. Then $F(X)=\mathbb{Z}$
but already $F^{(2)}(X)=F(\mathbb{Z})$ is the free group 
on an infinite number of generators. Consider the normal 
subgroup $2\mathbb{Z}\subset \mathbb{Z}$. Then
$F(2\mathbb{Z})\subset F(\mathbb{Z})$ is obviously not a normal
subgroup.
\end{example}

There are natural inclusion \textit{maps} 
$F^{(j-1)}(X)\subset F^{(j)}(X)$ defined
by the inclusion of the basis.
These maps are not homomorphisms
because the product of the elements $x_1,x_2$ in $F^{(j-1)}$ maps to the
basis element $x_1x_2$.
The inclusions combine to define the natural \textit{shift map}
$\widehat{\textrm{sh}}: \hat{F}(X)\rightarrow \hat{F}(X)$.
Note that each map $Y\rightarrow F(X)$ comes with the
usual extension, usually denoted by the same letter,
$F(Y)\rightarrow F(X)$
from the universal property of the free group functor.
Note that the extension of the inclusion $sh: F^{(j-1)}(X)\subset
F^{(j)}(X)$
is the identity homomorphism. 
Now consider a homomorphism $h: F(Y)\rightarrow F(X)$.
We can consider it as a map from the set $F(Y)$ into the group
$F(X)$ and extend to the homomorphism $F(F(Y))\rightarrow F(X)$.
By composition with the inclusion $F(X)\subset F^{(2)}(X)$ we have
defined the \textit{map} $h_{\bullet }: F^{(2)}(Y)\rightarrow F^{(2)}(X)$. 
Note that there is also the homomorphism $F(h): F^{(2)}(Y)\rightarrow
F^{(2)}(X)$
defined by mapping a word $w_1\ldots w_r$ with $w_i\in F(Y)$
to $h(w_1)\ldots h(w_r)$, which is different from $h_{\bullet }$. 

\begin{remark} The group $\hat{F}(X)$ has a second
natural product reminiscent of the multiplication
of polynomials. In fact if $(a_i),(b_j)$ are two elements in
$\hat{F}(X)$ let $(c_k)$ be the sequence defined by
$$c_k=\sum_{i+j=k}a_ib_j,$$
where the multiplication is performed in $F^{(k)}$ using the inclusion
maps $F^{(i)}\subset F^{(k)}$ and $F^{(j)}\subset F^{(k)}$.

Moreover the shift map defines a natural action of the group
$\hat{F}(X)$ on the set $\hat{F}(X)$ defined by 
$a\star b:=a\cdot \widehat{\textrm{sh}}(b)$ where the multiplication on the right hand
side is the usual multiplication in $\hat{F}(X)$
\end{remark}

There is a related natural construction. Note that the identity map
from the set $F(X)$ into the group $F(X)$ extends to the homomorphism
$F(F(X))\rightarrow F(X)$. By iteration there are naturally defined
homomorphisms $F^{(j)}(X)\rightarrow F(X)$ for each $j$. Thus we have
defined the homomorphism:
$$c: \hat{F}(X)\rightarrow \widetilde{F}(X):=\prod_{j=0}^{\infty}F(X).$$
Note that there is defined a natural shift homomorphism 
$\widetilde{\textrm{sh}}$ on $\widetilde{F}(X)$ such that
$c\circ \widehat{\textrm{sh}}=\widetilde{\textrm{sh}} \circ c$.
The shift homomorphism $\widetilde{\textrm{sh}}$ is a kind of
\textit{universal operator action}
representing the multiplication by a fixed ring element in the
commutative module case.

There is the unique natural \textit{functor}
$$\mathfrak{d}: \textrm{hom} \rightarrow F(\mathfrak{ob})$$
defined by
$$\mathfrak{d}(u)=\textrm{targ}(u)\textrm{sour}(u)^{-1},$$
where $F(X)$ is considered as a monoid.

A \textit{potential} of a $2$-groupoid is a
\textit{contravariant} functor
$$\mathfrak{a}: \textrm{hom} \rightarrow F(\mathfrak{ob}).$$ 
We will use $\mathfrak{a}$ to define inductively 
\textit{power series expansions} 
of objects in terms of models.
Note that $\mathfrak{a}$ uniquely extends to homomorphisms
$$F^{(j)}(\mathfrak{a}): F^{(j)}(\textrm{hom})\rightarrow
F^{(j+1)}(\mathfrak{ob})$$
for $j\geq 1$.
 
\begin{theorem} Let $\mathcal{C}$ be a $2$-groupoid
with models $\mathfrak{m}\subset \mathfrak{ob}$.
Let $\mathfrak{a}$ be a potential. Then there exists a 
natural normal subgroup $\mathcal{A}\subset \hat{F}(\mathfrak{m})$,
and the map
$$\rho : \mathfrak{ob}\rightarrow \hat{F}(\mathfrak{m})/\mathcal{A},$$
such that the induced homomorphism:
$$\rho: \hat{F}(\mathfrak{ob})\rightarrow
\hat{F}(\mathfrak{m})/\mathcal{A}$$
satisfies the formal skein relation:
$$\rho (\mathfrak{d}(u))=\widehat{\rm{sh}}(\rho (\mathfrak{a}(u^{-1})))$$
for all $u\in \rm{hom}_1$ and the induced map
$$\widehat{\text{sh}}: \hat{F}(\mathfrak{m})/ \mathcal{A} \rightarrow \hat{F}
(\mathfrak{m})/ \mathcal{A}.$$
\end{theorem}

\begin{proof}
The normal subgroup $\mathcal{A}$ and the homomorphism $\rho $
is constructed from a sequence of normal subgroups
$\mathcal{A}_j\subset F^{(j)}(\mathfrak{m})$ and homomorphisms
$\rho _j: \mathfrak{ob}\rightarrow
F^{(j)}(\mathfrak{m})/\mathcal{A}_j$. Let $\mathcal{A}_0$ be the
trivial group.
Define the map $\rho_0: \mathfrak{ob}\rightarrow
F^{(0)}(\mathfrak{m})$ by assigning to to each $x\in \mathfrak{ob}$
its unique model $m\in \mathfrak{m}$.
We define $\rho_1: \mathfrak{ob}\rightarrow
F^{(1)}(\mathfrak{m})/\mathcal{A}_1$ in the following way.
Consider $u\in \textrm{hom}(x,m)$. There is defined $\mathfrak{a}(u)\in
F(\mathfrak{ob})$. So we can apply the induced homomorphism
$$F(\rho_0): F(\mathfrak{ob})\rightarrow F^{(1)}(\mathfrak{m}).$$
Let $\lambda_1(u)$ be the resulting element.
Define $\mathcal{A}_1\subset F^{(1)}(\mathfrak{m})$ to be the normal
subgroup generated by all elements $\lambda_1(u)\lambda_1(v)^{-1}$ for
all $u,v\in \textrm{hom}(x,m)$. Then the image of $\lambda_1(u)$ in the quotient
$F^{(1)}(\mathfrak{m})/\mathcal{A}_1$ does not depend on $u\in
\textrm{hom}(x,m)$ but only on $x$. Next suppose that we have defined 
$\rho_j: \mathfrak{ob}\rightarrow F^{(j)}/\mathcal{A}_j$. Let
$\rho_j^1$ be a lift of $\rho_j$ to $F^{(j)}(\mathfrak{m})$.
Note that the set of lifts is an orbit of
of $\mathcal{A}_j$, and we can write $\rho_j^a$ for the lift differing
from $\rho_j^1$ by the action of $a\in \mathcal{A}_j$. For
$u\in \textrm{hom}(x,m)$ define 
$\lambda_{j+1}^a(u)\in F^{(j+1)}(\mathfrak{m})$ by
application of 
$$F(\rho_j^a): F(\mathfrak{ob})\rightarrow F(F^{(j)}(\mathfrak{m})=F^{(j+1)}(\mathfrak{m})$$   
to the element $\mathfrak{a}(u)$. 
Then define $\mathcal{B}_{j+1}$ by the normal subgroup of
$F^{(j+1)}(\mathfrak{m})$ defined by all elements
$$\lambda_{j+1}^{a_1}(u)(\lambda_{j+1}^{a_2}(v))^{-1}$$
for all $u,v\in \textrm{hom}(x,m)$ and all $a_1,a_2\in \mathcal{A}_j$.
Then the image $\rho_{j+1}(u)$ of $\lambda_1^a(u)$ in the quotient
$F^{(j+1)}(\mathfrak{m})/\mathcal{B}_{j+1}$ does not depend on the
choice of $a\in \mathcal{A}_j$ or $u\in \textrm{hom}(x,m)$ and $\rho_{j+1}$ is
well-defined. Finally define $\mathcal{A}_{j+1}$ by the normal
subgroup generated by $\mathcal{B}_{j+1}$ and all elements of the form
$\widehat{\textrm{sh}}(x^{-1})\widehat{\textrm{sh}}(xa)$ for all $x\in F^{(j)}(\mathfrak{m})$ and all
$a\in \mathcal{A}_j$. 
Then obviously $\widehat{\textrm{sh}}$ induces the map
$\widehat{\textrm{sh}}: \hat{F}(\mathfrak{m})/\mathcal{A} \rightarrow
\hat{F}(\mathfrak{m})/\mathcal{A}$.
The second claim follows comparing $\rho (\textrm{targ}(u))$ and $\rho
(\textrm{sour}(u))$. In the definition of $\rho $ any morphism can be chosen so
for a choice of $v\in \textrm{hom}(\textrm{targ}(u),m)$ in the first case we can choose
$vu\in \textrm{hom}(\textrm{sou}(u),m)$ in the second case where $m$ is the model for
$\textrm{targ}(u)$ and $\textrm{sour}(u)$. Then the equality follows immediately from
$\mathfrak{a}(vu)=\mathfrak{a}(u)\mathfrak{a}(v)$.  
\end{proof}

The subgroup $\mathcal{A}$ depends inductively on all possible choices of 
morphisms in the sets $\textrm{hom}(x,m)$. This is difficult to describe in
the most abstract setting, in particular because of the lack of
structure of the shift $\widehat{\textrm{sh}}$. 
Thus we consider the composition of $\rho $
with the projection onto $\widetilde{F}(\mathfrak{m})$. 

For each set $X$ and subset $U\subset \widetilde{F}(X)$ let
$(U)_{\widetilde{\textrm{sh}}}$ denote the smallest normal subgroup of 
$\widetilde{F}(X)$ containing all subsets $\widetilde{\textrm{sh}}^j(U)$ for all
$j\geq 0$.  

\begin{theorem}
The map $\rho $ induces the isomorphism
$$\widetilde{\rho }:
\widetilde{F}(\mathfrak{ob})/(\mathfrak{d}(u)\widetilde{\rm{sh}}(\mathfrak{a}(u^{-1}),u\in
\rm{hom})_{\mathfrak{\rm{sh}}} \cong \widetilde{F}(\mathfrak{m})/\widetilde{\mathcal{A}}.$$
\end{theorem}

\begin{proof} 
The homomorphism $c$ induces the epimorphism
$$\hat{F}(\mathfrak{m})/ \mathcal{A} \rightarrow
\widetilde{F}(\mathfrak{m})/
\widetilde{\mathcal{A}}$$
with $\widetilde{\mathcal{A}}:=c( \mathcal{A})$.
Thus by composition we have defined
$$\mathfrak{ob}\rightarrow
\widetilde{F}(\mathfrak{m})/\widetilde{\mathcal{A}}.$$
We extend this map to the homomorphism $\overline{\rho}$ defined on
$\widetilde{F}(\mathfrak{ob})=\prod_{j=0}^{\infty}F(\mathfrak{ob})$ by
mapping the sequence $(u_j)$ to $\prod_{j\geq
0}\widetilde{\textrm{sh}}^j(u_j)$.
This is an infinite product but finite in each factor of
$\widetilde{F}(\mathfrak{m})$, thus well-defined. Obviously
$$((\mathfrak{d}(u)\widetilde{sh}(\mathfrak{a}(u^{-1}),u\in
\textrm{hom}))\subset ker(\overline{\rho }).$$
Let $\widetilde{\rho }$ denote the homomorphism defined on the
quotient.
By construction $\overline{\rho }$ 
maps
$\widetilde{F}(\mathfrak{m})\subset
\widetilde{F}(\mathfrak{ob})$ by projection onto
$\widetilde{F}(\mathfrak{m})/\widetilde{\mathcal{A}}$. This implies
that $\widetilde{\rho }$ is onto.
Moreover the composition $\widetilde{\mathcal{A}}\subset
\widetilde{F}(\mathfrak{m})\subset
\widetilde{F}(\mathfrak{ob})\rightarrow
\widetilde{F}(\mathfrak{ob})/(\mathfrak{d}(u)\widetilde{\textrm{sh}}(\mathfrak{a}(u^{-1}),u\in
\textrm{hom})_{\widetilde{\textrm{sh}}}$ is the trivial homomorphism by the very construction of
$\mathcal{A}$.
This easily implies the injectivity of $\widetilde{\rho }$.
\end{proof}

\begin{remark} The group 
$\widetilde{F}(\mathfrak{ob})/((\mathfrak{d}(u)\widetilde{\rm{sh}}(\mathfrak{a}(u^{-1}),u\in
\textrm{hom}))$ is the completion of the 
\textit{$\mathfrak{a}$-skein group} defined by
$$\mathcal{S}(\mathfrak{a}):=\bigoplus_{j\geq 0}
F(\mathfrak{ob}))/(\mathfrak{d}(u)\textrm{sh}(\mathfrak{a}(u^{-1})))_{\rm{sh}},
$$
where $\rm{sh}$ is the shift homomorphism on the direct sum and 
$(\ )_{\rm{sh}}$ is defined as above.
\end{remark}

The advantage of working with $\widetilde{F}(\mathfrak{m})$ comes
from the easier description of the indeterminancy subgroup
$\mathcal{A}$ resulting from the homomorphism property of
$\widetilde{\textrm{sh}}$.
Note that the contributions to $\widetilde{\mathcal{A}}_{j+1}$
resulting from lower orders now can be 
describe by $\widetilde{\textrm{sh}}(\mathcal{A}_{j})$.

We will \textit{split} the contributions to $\mathcal{A}$ now in a natural way
associated to the two origins of relations. Actually such a splitting
is only possible in each degree but we can argue inductively.
The idea is that, because of the 
$2$-groupoid structure, we can consider all $u\in \textrm{hom}(m,m)$ with
$\textrm{mor}(u,1)\neq \emptyset$ and all those with $\textrm{mor}(u,1)=\emptyset$,
considered up to the action of $\textrm{mor}$. This suffices to dscribe 
$\mathcal{A}_1$ and thus by induction all $\mathcal{A}_j$.
Note that the change from $u: x\rightarrow m$ to $v: x\rightarrow m$
is provided only by $uv^{-1}$ and there is $h\in \textrm{mor}(u,v)$ if and
only if there is $h'\in \textrm{mor}(uv^{-1},1)$. 

Recall that the functor $\mathfrak{d}: \textrm{hom} \rightarrow
F(\mathfrak{ob})$ 
induces the homomorphism
$\mathfrak{d}: F(\textrm{hom} )\rightarrow F(\mathfrak{ob})$ by the free group
property. Also recall that $\rho_0: \mathfrak{ob}\rightarrow
F(\mathfrak{m})$ assigns to each object its model, and let $\rho_0$
also denote the extension to $F(\mathfrak{ob})$. 

\begin{definition} Let $\mathfrak{r} \subset \textrm{mor}$ be a subset.
An potential $\mathfrak{a}$ is called 
\textit{insensitive with respect to $\mathfrak{r}$}
if for each $m\in \mathfrak{m}$ and $r\in \textrm{hom}(m,m)$ such that
there exists $r\in \textrm{mor}(u,1)\cap \mathfrak{r}$, the following holds: 
(i) $\mathfrak{\rho}_0(\mathfrak{a}(u))=0$, and (ii)
there exists 
$w\in F(\textrm{hom})$ such that $\mathfrak{d}(w)=\mathfrak{a}(u)$ and
$\mathfrak{a}(w)=1$. The potential
is called \textit{insensitive} if it is
insensitive with respect to $\mathfrak{r}=\textrm{mor}$.
\end{definition}
 
\begin{theorem}
Suppose that $\mathfrak{a}$ is insensitive. Then
$\widetilde{\mathcal{A}}$ is the normal group generated by the images
of maps 
$$\rm{hom}(\it{m},\it{m})/\rm{mor}\rightarrow \widetilde{F}(\mathfrak{m})$$
for all $m\in \mathfrak{m}$
and all their shifts.
These maps are defined inductively by arbitrary expansion 
of elements in $\rm{hom}(\it{m},\it{m})$ using
the skein relation and following the idea of the  proof of Theorem 2.4. 
It suffices to
consider the image of a generating set of $\rm{hom}(\it{m},\it{m})/\rm{mor}$. 
\end{theorem}

\begin{proof}
Let $u\in \textrm{hom}(m,m)$ with $\textrm{mor}(u,1)\neq \emptyset$. 
The contribution from $u$ to $\mathcal{A}_1$ is provided by applying
$\rho_0$ from the proof of Theorem 2.4 (respectively its $c$-projection)
to $\mathfrak{a}(u)$. This involves a choice of homomorphism $w$ from
$\mathfrak{a}(u)$ to some element of $F(\mathfrak{m})$, such that
$\mathfrak{d}(w)=\mathfrak{a}(u)$. Note that because of $(i)$ of
insensitivity we know that the contribution in $\mathcal{A}_1$
vanishes. Moreover because of (ii) we know that the contributions in
higher order vanish too up to indeterminancies in order $\geq 2$.
(We also need that $\mathcal{A}_0$ is the trivial group). By induction
it can be proved that the contributions can be assumed to vanish
in all orders, up to the
indeterminancies resulting from the choices described by elements
of $u\in \textrm{hom}(m,m)$ with $\textrm{mor}(u,1)=\emptyset$.
\end{proof}

Consider $h\in \textrm{mor}(u,v)$ with $u,v: m\rightarrow m$. Let
$$d(h):=\mathfrak{a}(uv^{-1})\in F(\mathfrak{ob}).$$
Note that $d(h)$ actually only depends on its source and target.
Let $\mathcal{D}\subset F(\mathfrak{ob})$ be the normal closure 
of the image of $d$. Then there is well-defined homomorphisms
$$\gamma: \textrm{hom}(m,m)/mor\rightarrow F(\mathfrak{ob})/\mathcal{D},$$
which assigns to $u: m\rightarrow m$ the element 
$\mathfrak{a}(u)$. Here we use the group structure on
$\textrm{hom}(m,m)/\textrm{mor}$ given by $u\cdot v:=vu$, where 
on the right hand side the operation
is composition. The homomorphism $\gamma $ seems to be a kind of 
\textit{categorical} version of a string topology homomorphisms of
Chas and Sullivan. 

\vskip .1in

Next we will linearize the above structures.

\begin{definition} 
A \textit{model complexity} is a map
$$\mathfrak{c}: \mathfrak{m}\rightarrow \mathbb{N},$$
where $\mathbb{N}$ is the set of non-negative integers. 
\end{definition}

If $\mathfrak{m}$ is minimal then the model complexity first
extends to $\mathfrak{ob}$ by $\mathfrak{c}(x)=\mathfrak{c}(m)$, where
$m$ is the unique model of $x$. Then extend to $\textrm{hom}$ by defining
$\mathfrak{c}(u)=\mathfrak{c}(\textrm{targ}(u))=\mathfrak{c}(\textrm{sour}(u))$.
Finally define $\mathfrak{c}(h)=\mathfrak{c}(u)$ for each
$2$-homomorphism $h: u\rightarrow v$. Note that
$\mathfrak{c}(u)=\mathfrak{c}(v)$. There is always defined the trivial
complexity, which is a constant.
If there is no explicit mentioning of a complexity then we
will assume that we work with the trivial complexity.
Let $\mathfrak{ob}(j):=\mathfrak{c}^{-1}(j)$ and similarly for
$\textrm{hom}$ and $\textrm{mor}$.

\begin{definition} Let $R \supset \mathbb{Z}$ be a commutative ring with $1$. 

\noindent (i) A \textit{linear potential} is a contravariant functor:
$$\mathfrak{a}: \textrm{hom} \rightarrow R\mathfrak{ob},$$
where $R\mathfrak{ob}$ is considered as a groupoid (abelian group), 
and $\textrm{hom}$ is
the category with objects $\mathfrak{ob}$ and
morphisms $\textrm{hom}$.

\noindent (ii) Let $I\subset R$ be an ideal and $\mathfrak{c}$ be a
model complexity. Then $\mathfrak{a}$ 
is a linear potential with respect to $I$ if it
maps $\textrm{hom}(j)$ into the direct
sum of 
$$R( \bigcup_{i<j}\mathfrak{ob}(i))\bigoplus I( \bigcup_{
j\leq i\leq \ell}\mathfrak{ob}(i) )$$ for each $j\geq 0$
and a fixed number $\ell $ independent of $j$.  
\end{definition}

Let $\partial: R(\textrm{hom})\rightarrow R\mathfrak{ob}$ be defined by
$$\partial (u)=\textrm{targ}(u)-\textrm{sour}(u).$$

For a trivial model complexity we have a linear potential 
with respect to $I$ 
if it maps \textit{all} homomorphisms 
into $I(\mathfrak{ob})$.

Now replace in the constructions above
for each set $X$ the group product
$\widetilde{F}(X)$ by the infinite product of free modules
$$\prod_{j\geq 0}I^jX.$$
Replace the potential by the linear potential. 
Then the results above admit analogs in the commutative
framework of $R$-modules.
$$\prod_{i\geq 0}I^iX\cong RX[[I]].$$
The isomorphism is defined from the sequence of 
homomorphisms mapping 
$(u_i)\in \prod_{i\geq 0}I^iX$
to $\sum_{i\geq 0}u_i\in RX/I^jX$
for $j\geq 0$. This is compatible with projections and thus defines a
homomorphism into the inverse limit, which is easily be seen to be an
isomorphism. 
Similarly we call $\mathfrak{a}$ insensitive if for each $u\in \textrm{hom}(m,m)$
with $\textrm{mor}(u,1)\neq \emptyset$ we have that
(i) $\rho_0(\mathfrak{l}(u))=0$ and (ii) there exists $w\in \textrm{hom}$ such that
$\mathfrak{a}(u)=\mathfrak{d}(w)$ and $\mathfrak{a}(w)=0$.

\begin{example} 
Let $R$ be a commutative ring with $1$
and
let $I=(h)$. Let $X$ be a set. Let $\mathfrak{c}$ be the 
trivial complexity. 
Then there exists the natural \textit{linearization functor}
$$\widetilde{F}(X)\rightarrow RX[[I]]$$
by defining it on the $i$-th factor of $\widetilde{F}(X)$
such that $x\in X$ maps to $h^ix\in I^iX$. Note that this map
factors through $\mathbb{Z}(X)$ on each factor.
By composition with the linearization functor a potential becomes linear.
\end{example}

\begin{theorem} Let $\mathfrak{a}$ be a linear potential with
respect to $I$ and model complexity $\mathfrak{c}$. 
Suppose that $\mathfrak{a}$ is insensitive.
Then there
exist a map
$$\mathfrak{ob}\rightarrow R\mathfrak{m}/\mathcal{I},$$
where the submodule $\mathcal{I}$ is generated by the images
and $I$-translates of generators under maps
$$\rm{hom}(\it{m},\it{m})/\rm{hom}_2\rightarrow R\mathfrak{m}[[I]]\cong \prod_{i=o}^{\infty}I^i\mathfrak{m}.$$
Moreover there is an induced homormphism:
$$R\mathfrak{ob}/(\partial (u)-\mathfrak{a}(u),u\in \rm{hom})
\cong R\mathfrak{m}[[I]]/\mathfrak{I},$$
and an induced isomorphism
$$\left( R\mathfrak{ob}/(\partial (u)-\mathfrak{a}(u),u\in \rm{hom}) \right)[[I]]
\cong R\mathfrak{m}[[I]]/\mathfrak{I}.$$
\end{theorem}

\begin{proof} If $\mathfrak{c}$ is a trivial model complexity then the
proof is analogous to the proofs above of similar results.
Otherwise consider the subset $\mathfrak{ob}(i)\subset \mathfrak{ob}$. 
Then we will change the construction of Theorem 2.4 into the $k$-th factor 
in $\prod_iR\mathfrak{m}$. 
Let $\ell $ be the nonnegative integer defined in Definition 2.4.
Then for a given complexity $i$, all elements of $\mathfrak{ob}(i)$ map into
$I^{k-j}\mathfrak{m}$ for some $j\leq i+(k-j)\ell $. This can easily be proved by
induction. Thus $j(\ell +1)\leq i+k\ell $ and 
$$k-j\geq k-\frac{i+k\ell}{\ell +1}=:\phi(i,\ell,k).$$
The sequence of numbers $\phi (i,\ell ,k)$ converges to infinity for
$k$ to infinity. (Essentially we can replace the inverse limit
over $k$ by an inverse limit over $\phi (i,\ell ,k)$, see the remark
below).
Now one can
construct submodules $\mathfrak{I}_j$ and homomorphisms 
$$\mathfrak{ob}\rightarrow \bigoplus_{j\leq
k}R\mathfrak{m}/\mathfrak{I}_k=:\mathfrak{M}_k$$
for each fixed $i,k$. We may have to apply the
inductive procedure of Theorem 4 more than $k$ times to define
the image in $\mathfrak{M}_k$. In fact just apply it $k'$ times with
$\phi (i,\ell,k')\geq k$, but do the computation mod $I^{k+1}$.
\end{proof}

\begin{remark}
It is possible to construct submodules $\mathcal{C}_k\subset 
R\mathfrak{m}/I^k\mathfrak{m}$ and compatible maps
$$\mathfrak{ob}\rightarrow \left( R\mathfrak{m}/I^k\mathfrak{m} \right)
/\mathcal{C}_k$$
directly in the following way. 
The maps will be constructed by application of
the main procedure from Theorem 2.4, i.\ e.\ by replacing objects
by model objects, and elements in the image of $\mathfrak{a}$ based on
the \textit{skein relation} $\partial (u)=\mathfrak{a}(u)$ for $u\in \textrm{hom}$.
More precisely construct maps on $\mathfrak{ob}(i)$ for each $k$ by 
induction on the complexity $i$. Then we can assume that we have the 
expansion in $R\mathfrak{ob}$:
$$x=\sum_ja_jm_j+b+r,$$
with $a_j\in R, m_j\in \mathfrak{m}$ and $r$ a linear combintaion of
skein relations. Let $b=\sum_nb_nx_n$ with $b_n\in I^k$ and $x_n\in
\mathfrak{ob}$.  
Note that we can assume inductively that each object
in the linear combination $b$ has complexity at most $i+\ell k$. Now
apply the skein relations to $x_n$. 
This will give rise to a linear combination  
containing models, elements of the
form $ay$ with $a\in I$ and $y\in \mathfrak{ob}$ and elements of the
form $a'y'$ with $a'\in R$ but $c(y')\leq i+\ell -1$. Thus after at
most $\ell k+1$ applications of the skein relation we will have
written $x$ as a linear combination 
$$x=\sum_ja_j'm_j'+b+b'+r',$$
with $b\in I^{k+1}\mathfrak{ob}$, $c(b')<i$ and $r'$ a linear
combination of skein relations. Now we can apply the
inductive adssumption to expand $b'$ in the form we want, and finally
get an expansion we intend. The indeterminancies result from all
possible choices in the procedure of applications of the skein relation.
\end{remark}

We need one more definition for further discussions.

\begin{definition} We say that the $2$-groupoid is \textit{free on a set of
elementary homomorphisms $S$} if the set $\textrm{hom}$ is in one-to-one
correspondence with the \textit{subset} of $F(S)$ defined by the products of
composable words in $S^{\pm }$ such that the product of words
corresponds multiplication in the free group. In this case we write
$\textrm{hom}=F(S,\circ)$.
\end{definition}   

\section{Mapping spaces and Vassiliev groupoids}

General references for this section are \cite{B}, \cite{V1} and \cite{V2} for the topology and smooth structure of the mapping spaces. For some necessary results concerning transversality also see \cite{Wa2}. 

\vskip 0.1in

Let $M$ be an oriented compact $3$-manifold. Let $\textrm{Diff}(M)$ be the group
of diffeomorphisms of $M$, which are isotopic to the identity.
Let $\textrm{iso}(M)$ be the space of paths $\gamma : I\rightarrow \textrm{Diff}(M)$
with $\gamma (0)=id$. (There is the fibration
$$\Omega \textrm{Diff}(M)\rightarrow \textrm{iso}(M) \buildrel ev \over \longrightarrow \textrm{Diff}(M),$$
with $\textrm{ev}(\gamma )=\gamma (1)$, which is closely related to the $2$-groupoids discussed in this section.)

\vskip .1in

In this section we describe the $2$-groupoids used in this
paper.
Our approach could generalize to spaces $X$ 
with a suitable manifold stratification in codimensions
$\leq n$ (generalizing spaces of immersions of circles into a
$3$-manifold). In fact the easiest application here could be
to spaces of immersions of circles into higher dimensional manifolds.
For different results relating stratified spaces and $2$-catgeory theory see \cite{T}.

\vskip .1in

The idea is to replace the fundamental $n$-groupoid by a
subgroupoid consisting of maps transversal with respect to the
stratification. The resulting $n$-groupoid is \textit{quantized}
with respect to certain morphisms, i.\ e.\ we form a quotient groupoid
in which these morphisms are identities. For $n=2$ and $X$ a space of
immersions of circles into
the $3$-manifold $M$, the quanitzation is 
with respect to $1$-morphisms, which are isotopies (traces of the action
of $\textrm{Diff}(M)$). Here we only describe the resulting $2$-groupoid. 
It turns out that forming the quotient groupoid is a quite delicate problem on its own.
 
\vskip .1in

Let $\widetilde{\textrm{imm}}(j)$ be the set of \textit{singular links}
$f: \cup_{j}S^1\rightarrow M$ as defined for knots in \cite{B}, chapter 3. 
These are immersions, which induce embeddings on $\cup_jS^1\setminus A$
for a finite subset $A$ of $\cup_jS^1$ and for any two distinct points $x_1,x_2\in A$ the branches 
of $f$ at $f(x_1)$ and $f(x_2)$ can have tangency of at most finite order.
We will also assume that all small diagonals are deleted (i.\ e.\ immersions containing identical
components). It is shown in \cite{B} that $\widetilde{\textrm{imm(j)}}$ is an open subset of the
space of all immersions $\cup_jS^1\rightarrow M$, where the last space is equipped with a smooth 
structure based on an infinite limit of Hilbert spaces.  Moreover the proof in (\cite{B}, 3.2.1 Lemma)
generalizes easily from knots to links showing that the inclusion of singular links into the space of 
all immersions is a weak homotopy equivalence.

Note that the diagonals form a subset of infinite
codimension which can be ignored in most discussions.
The symmetric group $\Sigma_j$ acts smoothly and freely on $\widetilde{imm}(j)$ by permuting
the circles in the domain. Let $\textrm{imm}(j)$ be the resulting quotient manifold.

For each space $X$ let $C_i(X)$ denote the space of ordered
configuations of $i$ points in $X$. An ordered $i$-configuration is just
an embedding of the space $*_1\cup \ldots \cup *_i\hookrightarrow X$. 
Let $k\geq 1$. Consider the evaluation map
$$ev: \widetilde{\textrm{imm}}(j)\times C_{2k}(\cup_j S^1)\rightarrow M^{2k}$$
defined by $\textrm{ev}(f,x):=f(x)$.
Let $\Delta \subset M^{2k}$ be the submanifold of those points
$(x_1,y_1,x_2,y_2,\ldots ,x_k,y_k)$ 
with $x_i=y_i$ for $i=1,\ldots k$ but $x_i\neq x_j$ for $i\neq j$.
Note that $\Delta $ is a locally closed submanifold of
$M^{2k}$, and is diffeomorphic to the configuration space
$C_{k}(M)$ of ordered $k$-configurations in $M$.

Let $\widetilde{\textrm{imm}}[k](j):=ev^{-1}(\Delta )$. 
It follows e.g. from \cite{Gl} and standard \textit{local} arguments that this is an infinite dimensional submanifold of 
$\widetilde{\textrm{imm}}(j)\times C_{2k}(\cup_j S^1)$ of codimension $k$. 
The symmetric group acts on 
$\widetilde{\textrm{imm}}(j)\times C_{2k}(\cup_jS^1)$ by simultaneously permuting
the circles in the domain of immersions and the circles in the
codomain of the embeddings in $C_{2k}(\cup_jS^1)$.
The group $\mathbb{Z}_2^k$ acts by permuting 
the points in the domain of embeddings $*_1\cup \ldots
*_{2k}\rightarrow \cup_jS^1$.
Let $\textrm{imm}[k](j)$ be the quotient by the actions. (Again we may delete
diagonals as above to consider a free action.)

Now for $k\geq 0$ let
$$\textrm{imm}[k]:=\bigcup_{j\geq 0}\textrm{\textrm{imm}}[k](j)$$
be the space of \textit{$k$-immersions} in $M$.
Note that a $k$-immersion always has $k$ distinguished \textit{ordered} 
double-points.

\begin{remark} 
The spaces $\textrm{imm}[k](j)$ are related to spaces of 
mappings of graphs into $M$. In fact, one can use the ordering of 
the configuration space coordinates, orientation and minimal distance 
between the $2k$ points in $\cup_jS^1$ to define a map from 
$\cup_jS^1$ onto oriented standard graphs, such
that immersions with the given configuration factor through 
mappings of graphs. This can be done \textit{continously} over $\textrm{imm}[k](j)$. The graphs
are equipped with an orientation around the vertices and basepoints
on components away from the vertices. This will define a continuous
mapping from $\textrm{imm}[k](j)$ onto a space of mappings of graphs into $M$.
Using length parameters on edges of the graphs it is possible
to construct an inverse map. 
\end{remark}

Let $\textrm{emb}[k]\subset \textrm{imm}[k]$ be the subset representing immersions with
precisely $k$-double-points without tangencies. These immersions are
called $k$-embeddings. 

For $k\geq 0$ let $\mathcal{L}[k]$ denote the set of \textit{isotopy} classes
of $k$-embeddings. Here two $k$-embeddings are isotopic if there is a
continous path in $\textrm{emb}[k]$ joining the two $k$-embeddings. 

Then $\mathcal{L}[0]$ is the usual set of
isotopy classes of oriented embeddings of circles also called oriented
links in $M$. The elements of $\mathcal{L}[k]$ in general are called
oriented $k$-links in $M$.

Let $\textrm{imm}[k,1]$ denote the subspace of elements in $\textrm{imm}[k]$ 
representing immersions with precisely $(k+1)$ double-points 
without tangencies. 

\begin{lemma} The space $\rm{imm}[\it{k},\rm{1}]\subset \rm{imm}[\it{k}]$ is naturally homeomorphic to
the space $\rm{emb}[\it{k}+\rm{1}]$.
\end{lemma}

\begin{proof}
There is a natural ordering of the double-points by putting the
distinguished double-point in the last place. Thus the sets are in
one-to-one correspondence, which is easily seen to be a homeomorphism.
\end{proof}

Let $\textrm{imm}[k,2]$ denote the union of the set of elements in $\textrm{imm}[k]$ 
consisting of (i) immersions with $(k+2)$ double-points without tangency, 
(ii) immersions with $(k-1)$ double points without tangency and a triple
point, and (iii) immersions with $k$ double-points, $(k-1)$ of which
without tangency and a tangency at one of the double-points.  
Note that the immersions  
with (i) in $\textrm{imm}[k,2]$ have an ordering of
double-points up to permutations of the last two double-points.
Thus there is a natural $2$-fold covering map $\textrm{emb}[k+2]\rightarrow 
\textrm{imm}[k,2]_{\bullet}$, where $\textrm{imm}[k,2]_{\bullet} \subset \textrm{imm}[k,2]$ is
the subspace satisfying (i). In general let $\textrm{imm}[k,j]_{\bullet}\subset
\textrm{imm}[k]$ be the subspace represented by immersions, whose singularities are
the $k$ ordered double-points without tangency, and $j$
additional double-points without tangency. 

\vskip .1in

We will now describe the sequence of $2$-groupoids $\mathcal{C}[k]$ for
$k\geq 0$. Skein theory and Vassiliev theory should be considered as a
construction on these groupoids.

\vskip .1in

The set of objects of $\mathcal{C}[k]$ will be the 
set $\mathcal{L}[k]$ of singular links with $k$ double points, 
i.\ e.\ isotopy classes
($\textrm{Diff}(M)$-orbits)
of immersions of circles in $M$ with precisely $k$ double points
without tangencies.   

\vskip .1in

The $1$-morphisms in the category $\mathcal{C}[k]$ will be the
\textit{isotopy classes of transversal
paths} in $\textrm{imm}[k]$. 
A path $\gamma: I\rightarrow \textrm{imm}[k]$
is \textit{transversal} if $\gamma (t)\in \textrm{emb}[k]$ except for finitely many 
$t$ where $\gamma (t)$ maps to an immersion with precisely 
$(k+1)$ double points without tangencies. Moreover, in the neighbourhood
of such $t$ the image of $\gamma (s)$ runs through a standard crossing
change in some oriented $3$-ball in $M$. 
Note that a homotopy between two elements of $\textrm{emb}[k]$ can be approximated by 
a transversal homotopy. The set of points in $\textit{imm}[k]$ with
a tangency at one of the $k$ double points has codimension $2$ in 
$\textrm{imm}[k]$, also compare \cite{L}, 3.2. In \cite{L} Lin works in the PL-setting.
The general reference for this transversality arguments and those in the following arguments we 
refer to \cite{V1}. (Here the mapping spaces are replaced by finite dimensional 
approximations using an embedding $M\subset \mathbb{R}^N$ for some $N$ and 
considering \textit{projections} of polynomial approximations in a tubular neighborhood of the embedding.)
Vassiliev's theory will allow us to focus on the codimension of singularities and apply transversality like 
in a finite dimensional setting.  

\vskip .1in

Next we define \textit{isotopy} of transversal paths.
Let $\chi: I\rightarrow \textrm{iso}(M)$ be a \textit{piecewise} continuous path. 
Assume that $t\mapsto \chi (t)(1) \in \textrm{Diff}(M)$
is continuous. Note that $\chi (t)(0)=id$ for all $t$.
The \textit{adjoint} of $\chi $ is the map
$\chi ': I\times I\rightarrow \textrm{Diff}(M)$
defined by $\chi '(t,t')\rightarrow \chi (t)(t')$. 
Then $\chi $ acts as a \textit{deformation}
$I\times I\rightarrow \textrm{imm}[k]$ of $\gamma $ by
$$(t,t')\mapsto \chi '(t,t')\circ \gamma (t).$$
 
For a given transversal path $\gamma $ we assume that $\chi $ only 
jumps at parameters $t$ away from
the singular parameters of $\gamma $. 
We assume that $\chi $ is a jump function in the following way.
At a jump parameter $t$ the 
limit of the isotopies
from the left or right is different from the
isotopy at $t$. But we always assume continuity from left or right.
Then we say that $\chi $ defines the \textit{isotopy} 
$\gamma \rightarrow \gamma '$ with
$\gamma '(t)=\chi (t)(1)\circ \gamma (t)$. The singular 
parameters of $\gamma '$ and $\gamma $ coincide. In this way isotopy 
can change $\gamma $ between singular parameters by arbitrary loops 
in $\textrm{emb}[k]$. 
We will allow a second class of isotopies of a more trivial nature. 
Two transversal paths $\gamma ,\gamma '$ are also isotopic if there
is a transversal homotopy $I\times I\rightarrow \textrm{emb}[k]\cup imm[k,1]$
restricting to $\gamma $ respectively $\gamma '$ on
$I\times \{0 \}$ respectively $I\times \{1 \}$. Transversality of
the homotopy means that the preimage of the set $\textrm{imm}[k]\setminus
\textrm{emb}[k]$ is a $1$-dimensional manifold properly embedded in $M$. 

\vskip .1in

Note that composition of $1$-morphisms is well-defined. In fact, let
$\gamma : K_1\rightarrow K_2$ and $\gamma ':K_2\rightarrow K_3$ 
are homomorphisms for singular links $K_i$, $i=1,2,3$. Here we choose
representative paths $\gamma, \gamma '$ between corresponding
representing immersions. We can choose any isotopy between
representatives $\gamma (1), \gamma '(0)$ of $K_2$, and take 
the usual composition of paths.
The result of the composition does not depend on the choice
of isotopy up to isotopy of transversal paths.
The composition $\circ $ of transversal paths is associative because 
isotopy includes reparamatrizations of the interval.
But in general $\gamma \circ \gamma^{-1}$ is not isotopic to the constant
path because the embedding of the singular parameters in $I$ is
well-defined up to isotopy.
 
\vskip .1in

Finally the $2$-morphisms between any two isotopy classes of 
transversal paths will be defined by \textit{homotopy classes} of 
\textit{transversal homotopies} 
between representing paths in
$\textrm{imm}[k]$. The appropiate notion of homotopy class is defined
below. 
There are obvious ways of horizontal and vertical 
composition of maps $I\times I\rightarrow \textrm{imm}[k]$ with certain boundary
restrictions coinciding (corresponding to $\circ_2$ and $\circ _1$.
A \textit{transversal homotopy} is the result of
\textit{horizontal and vertical composition} of 

(i) continous maps
$$I\times I\rightarrow \textrm{imm}[k],$$
which are constant on $I\times \{0,1 \}$, and

(ii) maps $I\times I\rightarrow \textrm{imm}[k]$, which are adjoint  
to piecewise continuous paths $I\rightarrow \textrm{iso}(M)$ as defined above. 
(in general this is not constant on $I\times \{0,1 \}$, but the
restriction to $I\times
\{0,1 \}$ is two continuous paths in $\textit{emb}$), and

(iii) transversal maps $I\times I\rightarrow \textrm{emb}[k]\cup \textrm{imm}[k,1]$.

\vskip .1in

The transversality of $2$-morphisms will be in the
spirit of Lin \cite{L}. We can assume that the preimage of the set of 
those elements in $imm[k]$, which are not in $emb[k]$, is a graph with
vertices of valence $1$, $2$ or $4$. All vertices of valence $1$ will
be contained in the boundary $\{0,1 \}\times I$ and be mapped to immersions
with precisely $k+1$ double-points without tangencies. 
All the open edges of the graph are mapped to immersions with precisely 
$k+1$ double-points without tangencies. 
The vertices of valence $4$ or $1$ are contained in the interior 
of $I\times I$.
Those of valence $4$ are mapped to either (i) immersions with $k+2$
double-points without tangency, or (ii) immersions with one
transversal triple
point (with a distinguished branch) and $k-1$ double points 
without tangency. The vertices of valence $2$ are mapped to immersions 
with $k$ double points but with a tangency at
one of those double-points. 
Note that for $k=0$ both the tangency contributions and triple points
vanish.
   
Now we explain 
homotopy of transversal homotopies. 
Note that a $2$-morphism has source and target
\textit{isotopy} classes of transverse paths.
We will generate homotopy of $2$-morphisms
by homotopy of (i), (ii) and compositions $\circ _1,\circ_2$.
A homotopy of (i) is a continous map
$$(I\times I)\times I\rightarrow \textrm{imm}[k],$$
which is constant on $(I\times \{0,1\})\times I$.
Finally we \textit{define} any two transversal homotopies of type (ii)
mapping the transversal path $\gamma $ 
to the transversal path $\gamma '$ to be homotopic to the identity 
on the class of $\gamma $.
It is now easy to see that the set of $2$-morphisms $mor$ is
well-defined. Note that the defintions are straightforward in order
to have well-defined
compositions operations on the quotient groupoid.

\begin{lemma} Each homotopy between transversal paths joining two
$k$-embeddings can be approximated relative to the boundary by a
transversal homotopy.
\end{lemma} 

\begin{proof}  We consider the finite dimensional approximation
model for the space of immersions respectively singular links, see \cite{V1},\cite{V2}.
It can be assumed that evaluation maps restricted to
$\textrm{ev}^{-1}(\Delta )\subset \widetilde{\textrm{imm}}(j)\times C_{2k}(\cup _jS^1)$
corresponding to higher order singularities and jet singularities
are transversal.
Then the result follows from 
known local models of the singularity strata and their codimensions 
in these spaces, for details see \cite{V2}. 
Obviously the open stratum in $\textrm{imm}[k]$ is $\textrm{emb}[k]$,
the codimension-$1$ stratum consists of immersions with precisely
$(k+1)$ ordered 
double-points without tangency.
The codimension $2$-stratum 
is constructed from the codimension-$1$ generic 
degeneracies in the limit set of the immersions in the 
codimension-$1$ stratum.     
\end{proof}

\begin{remark}

\noindent (i) Lin's transversality results \cite{L} are the piecewise
linear versions of the results of the lemma, see also the results by
Stanford \cite{S}, Bar-Natan and Stoimenov \cite{BS} and Hutchings \cite{H}.

\noindent (ii) The result of the lemma holds for any map of a surface
$F\rightarrow \textrm{imm}[k]$ with given transversal map on the boundary 
$\partial F$.

\noindent (iii) It is easy but tedious to prove the groupoid property of the $2$-category $\mathcal{C}[k]$.
The inverse of a transverse path is the inverse path $\gamma^{-1}$ with $\gamma^{-1}(t)=\gamma (1-t)$. 
The homotopy $I\times I\rightarrow \textrm{emb}[k]\cup \textrm{imm}[k,1], (s,t)\mapsto \gamma (t)$ then is 
transversal according to the definition above and shows $\textrm{hom}([\gamma ^{-1}]\circ [\gamma ],*_1)\neq \emptyset $ for the constant path $*_1$ at $\gamma (0)$, after suitable reparametrization. Similarly we deduce from this $\textrm{hom}([\gamma ]\circ [\gamma^{-1}],*_2)\neq \emptyset $ for $*_2$ the constant path in $\gamma (1)$. The inverse of a $2$-morphism represented by a transversal homotopy $h: I\times I\rightarrow  \textrm{emb}[k]\cup \textrm{imm}[k,1]$ is similary defined by the transversal homotopy $h^{-1}(s,t)=h(s,1-t)$. Note that in this case the composition represents the corresponding identity because $2$-morphisms are by definition homotopy classes.
Some care has to be taken of the fact that $1$-morphisms are not necessarily represented by homotopic paths because corresponding transversal path representatives can differ by nontrivial loops in $\textrm{emb}[k]$.  

\end{remark}

\begin{definition} The \textit{graded} $2$-groupoid
$$\mathcal{C}=\mathcal{C}(M):=\bigcup_{k\geq 0}\mathcal{C}(k)$$
is called the \textit{homotopy Vassiliev groupoid} of $M$.
\end{definition}

The following two basic results describe the structure of the
$2$-groupoids $\mathcal{C}[k]$. The first theorem is immediate from
the definitions. The proof of the second one follows from Lemma 3.1.

\begin{theorem}
For each $k\geq 0$ the morphism set $\rm{hom}[\it{k}]$ is 
$\circ $-generated by the elementary morphisms $K_+\rightarrow K_-$
defined by crossing changes, and their inverses $K_-\rightarrow K_+$,
for all $K_*\in \mathcal{L}[k+1]$ with $*$ indicating the last
double-point. By identification of $\mathcal{L}[k+1]$ with the
elementary morphisms in $\mathcal{C}[k]$ we have
$$\rm{hom}[\it{k}]=F(\mathcal{L}[k+1],\circ).$$ 
\hfill{$\square$}
\end{theorem}

Note 
that for any $K\in \mathcal{L}[k]$ there is the
unique $1$-morphism $1$, which is represented by \textit{any}
isotopy (path in $\textrm{emb}[k]$)
between representing
embeddings of $K$. 

We describe resolutions of triple points. Consider immersions
in $imm[k]$ with precisely $(k-1)$ double points without tangencies and
a transverse triple point. There is a distinguished branch and an
ordering of the remaining two branches at the triple point. Let
$\mathcal{T}[k]$ denote the set of isotopy classes of these
immersions. Let $K\in \mathcal{T}[k]$. We will define resolutions
$K_{\pm}^{1/2}\in \mathcal{L}[k+1]$ in the following way: $K_{\pm
}^{1/2}$ ist the positive respectively negative resolution of the
double point, which appears when the first respectively second branch
is ignored and only the remaining and the distinguished branch are
considered to intersect at the triple point.  
The geometric $4T$-relation will be built with these singular links.
See \cite{BS}, \cite{H} and \cite{V1} for the description of geometric $4T$-relations.

\begin{theorem} 
The morphism set $\rm{mor}[\it{k}]$ is generated via compositions 
$\circ _1$ and $\circ _2$
by elementary $2$-morphisms of the following form:

\begin{enumerate}

\item[\textup{(i)}] morphisms
$K_{-*}^{-1}K_{*+}^{-1}K_{+*}K_{*-}\rightarrow 1$ ($1$ is the
1-morphism for $K_{--}$) defined for each $K_{**}\in \mathcal{L}[k+2]$.
(differentiablity relations)

\item[\textup{(ii)}] for $k>0$, morphisms
$(K_+^1)^{-1}(K_+^2)^{-1}K_-^1K_-^2\rightarrow 1$
defined for each $K\in \mathcal{T}[k]$
(geometric 4T-relation).

\item[\textup{(iii)}] for $k>0$, morphisms $K\rightarrow K'$ for $K,K'\in
\mathcal{L}[k+1]$ defined by changing the order of the 
two double-points appearing in the natural two-point resolution of
a tangency.  
\end{enumerate}
\end{theorem}

\begin{proof}
The result is immediate from the definitions and Lemma 3.2.
\end{proof}

\begin{remark}
It is possible to define a slight variation of the
$2$-groupoid $\mathcal{C}$ by changing the definition of isotopy 
of $1$-morphisms given above in the following way. Instead of allowing
arbitrary transversal homotopies in $\textrm{emb}[k]\cup \textrm{imm}[k,1]$ only allow
homotopies defined by reparametrizations of the interval (use
homotopies from $\textrm{id}$ into monotone functions $I\rightarrow I$).
Then of course we have to introduce the isotopies of the above form
into the $2$-morphisms. In this case the $1$-morphisms $\textrm{hom}[k]$ will
be identified with the \textit{monoid} generated by
$\mathcal{L}[k+1]^{\pm 1}$. Then in the statement of Theorem 3.2
we have introduce additionally $uu^{-1}\leftrightarrow 1$. 
\end{remark}

We can construct a set of models $\mathfrak{m}$ for the category 
$\mathcal{C}$ by choosing an element in $\textrm{emb}[k]$ for each 
path component of the spaces $\textrm{imm}[k]$ and $k\geq 0$. 
Note that the set of path-components of $\textrm{imm}[k]$ is the set
of homotopy classes of singular $k$-links. 
$\mathfrak{b}=\cup_{k\geq 0}\mathfrak{b}[k]$.
Note that $\mathfrak{b}[0]$ is in one-to-one correspondence with the
set of monomials in $\hat{\pi}$, where $\hat{\pi}$ is the set
of conjugacy classes of $\pi_1(M)$.

For $k\geq 1$, the homotopy classes of singular $k$-links are in
one-to-one correspondence with the set of chord diagrams in $M$ with
$k$ chords. These are usual chord diagrams equipped with free homotopy
classes of maps into $M$ assigned to each component of the complement of the
set of endpoints of chords in $\cup_jS^1$. We will require an ordering
of the chords for each representative chord diagram. Note that it is possible
that chord diagrams with different orderings are homotopic, depending
on the topology of the space of immersions into $M$.    

\vskip .1in

The differentiablility relations
impose \textit{commutativity on the level of models}. This means that
indeterminancies of expansions defined by abstract skein potentials 
in $\tilde{F}(\mathfrak{ob})$
in section 2 contain the commutators of $\mathfrak{ob}$.
Thus it suffices to consider skein
potentials like in the usual skein or Vassiliev theory.

\vskip .1in

Next we will define a certain $2$-groupoid 
$\mathcal{B}$ from suitable bordism classes of mappings 
of $i$-dimensional manifolds into $\textrm{imm}$ for $i=0,1,2$.
In this category
the $1$-morphisms are \textit{linearized}
in a certain way, as is suggested by the notion of skein potential
in a commutative ring with $1$.
The precise definition of this $2$-groupoid, called the 
\textit{homology Vassiliev groupoid} turns out to be quite
interesting and subtle on its own.  
In the following note that bordism and homology coincide in dimensions
$\leq 2$, see (\cite{Kau}, page 319)

\vskip .1in

The objects of $\mathcal{B}[k]$ are the same as those of
$\mathcal{C}[k]$,
i.\ e.\ the isotopy classes of embeddings into $\textrm{emb}[k]$.

\vskip .1in

The $1$-morphisms $x\rightarrow y$, for $x,y\in \mathcal{L}[k]$, 
are the oriented bordism classes $u$ of maps of oriented compact $1$-manifolds
$W\rightarrow \textrm{emb}[k]\cup \textrm{imm}[k,1]$, i.\ e.\ elements of
$H_1(\textrm{emb}[k]\cup \textrm{imm}[k,1],\textrm{emb}[k])$ such that 
$\partial (u)=y-x$, where
$$\partial: H_1(\textrm{emb}[k]\cup \textrm{imm}[k],\textrm{emb}[k])\rightarrow H_0(\textrm{emb}[k])$$
is the usual boundary operator. Note that this is well-defined
since $H_0(\textrm{emb})$ is the free abelian group on $\mathcal{L}[k]$.
Thus the set of $1$-morphisms $\textrm{hom}$ of $\mathcal{B}$ is a certain
\textit{subset} of the homology group. Note that $W$ can be
represented by a map
$$(I\cup \cup_iS^1,\partial I) \rightarrow (\textrm{emb}[k]\cup
\textrm{imm}[k],\textrm{emb}[k]).$$
There can be \textit{arbitrary} maps of closed components
into $\textrm{emb}[k]\cup \textrm{imm}[k,1]$. 
Note that a representative map may very well contain other component
maps $I\rightarrow \textrm{emb}[k]\cup \textrm{imm}[k,1]$. But the homological boundary
of all these components vanishes in $\textrm{emb}[k]$. 

\vskip .1in

The $2$-morphisms $u_1\rightarrow u_2$ for $1$-morphisms 
$u_i: x\rightarrow y$ and $i=1,2$, are more difficult to describe.
See (\cite{Wa}, page 6) for a description of the general bordism set-up used
here.
We consider \textit{bordism classes} of quadruples of maps ($j=1,2$)
$$g: (F,\partial_jF,\partial_1F\cap \partial_2F)
\rightarrow (\textrm{imm}[k],\textrm{emb}[k]\cup
\textrm{imm}[k,1],\textrm{emb}[k]),$$
where $F$ is an oriented compact surface with boundary 
$\partial F=\partial_1F\cup \partial_2F$ such that
$\partial_1F\cap \partial_2F$ is a disjoint union of two 
points. We require that $g|(\partial_jF)$ represents 
$u_j\in H_1(\textrm{imm}[k]\cup \textrm{emb}[k],\textrm{emb}[k])$ for $j=1,2$.
Suppose that we have given two quadruples $g_i$, $i=1,2$, as above:
$$g_i: (F_i,\partial_jF_i,\partial_1F_i\cap \partial_2F_i)\rightarrow 
(\textrm{imm}[k],\textrm{emb}[k]\cup \textrm{imm}[k,1],\textrm{emb}[k]).$$
A bordism from $g_1$ to $g_2$ as above is a quadruple of maps 
$$G: (W,\partial_jW,\partial_1W\cap \partial_2W)\rightarrow
(\textrm{imm}[k], \textrm{emb}[k]\cup \textrm{imm}[k,1],\textrm{emb}[k]),$$
where $W$ is an oriented compact $3$-manifold with corners.
We have that $\partial_j(W)\subset \partial W$ 
is a $2$-manifold with corners for $j=1,2$.
More precisely $\partial W=\partial_1 W\cup \partial_2W\cup F_1\cup
F_2$, $\partial _jW\cap F_i=\partial_jF_i$ for $i,j=1,2$, 
$\partial_jW$ is a bordism between $\partial_jF_1$ and 
$\partial_jF_2$, and $\partial_1W\cap \partial_2W$ is a bordism from
$\partial_1F_1\cap \partial_2F_1$ to $\partial_1F_2\cap
\partial_2F_2$.
The map of $W$ restricts to the maps given by $g_i$ 
on the corresponding strata.

\vskip .1in

The definition of horizontal and vertical compositions of
$2$-morphisms requires cut, paste and smoothing arguments but is straightforward.

\vskip .1in

There is an obvious surjection ($\textrm{mor}$ in the category $\mathcal{B}$)
$$\textrm{mor} \rightarrow H_2(\textrm{imm}[k],\textrm{emb}[k]\cup \textrm{imm}[k,1]),$$
where we identify elements in $H_2(\textrm{imm}[k],\textrm{emb}[k]\cup \textrm{imm}[k,1])$ with
bordism classes of maps of oriented compact surfaces
$$(F,\partial F)\rightarrow (\textrm{imm}[k],\textrm{emb}[k]\cup \textrm{imm}[k,2]).$$

\vskip .1in

The proof of the next result is immediate. We omit the tedious details because we will not use this result in the rest of the article. 

\begin{proposition}
There exists the natural forgetful functor
$\mathcal{C}(M)\rightarrow \mathcal{B}(M)$. \hfill{\small{$\square$}}
\end{proposition}

Obviously we
can use the images of the elementary category generators of $\mathcal{C}$ 
for the category $\mathcal{B}$.
There has to included the attaching of certain handles. We will
not describe this in detail at this point, see the next remark.

\begin{remark} The $2$-groupoid $\mathcal{B}$ is a \textit{linearization} of 
$\mathcal{C}$ in the following way: 
Let $\textrm{\textrm{hom}}[k]$ denote the set of $1$-morphisms in the
skein groupoid. Then $\textrm{hom}[k]\subset
\mathbb{Z}\mathcal{L}[k+1]$
using the obvious identification. 
This is e.\ g.\ contained in the proof of Proposition 4.1 in the next
section.
It also follows from Remark 3.2 (ii). 
The additional commutation relations are of course induced by suitable
maps of tori into $\textrm{imm}$.
We have $\textrm{hom}[k]\neq
\mathbb{Z}\mathcal{L}[k+1]$.
The composition of elements in $\textrm{hom}[k]$ corresponds to the usual addition
of homology classes. But it is only defined for classes 
$u,v\in H_1(\textrm{emb}[k]\cup \textrm{imm}[k,1])$, for which $\partial (u)=z-y$,
$\partial (v)=y-x$ such that
$\partial (u+v)=z-x$. Then $u\circ v$ is a morphims from $x$ to $z$
and is represented by the homology class $u+v$.
We will write $\textrm{hom}[k]=A(\mathcal{L}[k+1],\circ)\subset
\mathbb{Z}\mathcal{L}[k+1]$ to indicate that the $1$-morphisms are
the subset of the free abelian group corresponding to 
\textit{composable} morphisms of the $2$-groupoid.
Note that by the excision property of homology the explicit insertion
of loops in $emb$ is not necessary in the category $\mathcal{B}$.
\end{remark}

\vskip .1in

In the next sections we will study the structure
of the skein groupoid and its variations 
relevant for skein modules.
We will see that all the necessary facts needed 
are consequences of the exact homology
sequences of the pairs $(\textrm{imm}[k],\textrm{emb}[k])$.
More details about $\mathcal{B}$ will be given in a future article. 

\vskip .1in

In order to keep notation short we will consider the graded spaces
$\textrm{imm}=\cup_{k\geq 0}\textrm{imm}[k]$. We let 
$\mathcal{L}:=\cup_{k\geq 0}\mathcal{L}[k]$. 
We use the shift notation $\mathcal{L}[+1]$ and define
$\mathcal{L}[+1][k]:=\mathcal{L}[k+1]$. 

\vskip .1in

\begin{definition} Let $\mathcal{R}$ be a commutative ring with $1$.
A \textit{skein potential in $\mathcal{R}$} is a map
$$\sigma: \mathcal{L}[+1]\rightarrow \mathcal{R}\mathcal{L}.$$
Sometimes we will only consider 
$$\sigma: \mathcal{L}[1]\rightarrow \mathcal{L}[0]$$
and call this also a skein potential.
\end{definition}

\begin{lemma} Each skein potential defines a linear potential for the
$2$-categories $\mathcal{C}$ and $\mathcal{B}$.
\end{lemma}

\begin{proof} This follows from the description of the $1$-morphisms
in $\mathcal{C}$ and $\mathcal{B}$ respectively. In fact, the skein potential extends to a 
homomorphism of abelian groups $\mathbb{Z}\mathcal{L}[+1]\rightarrow
\mathcal{R}\mathcal{L}$. But $\textrm{hom} \subset \mathbb{Z}\mathcal{L}[+1]$
and the inclusion maps composition into sum. 
\end{proof}

Note that a skein potential defined $\mathcal{L}[1]\rightarrow
\mathcal{R}\mathcal{L}[0]$ can easily be extended trivially to a full
skein potential. So usually we will not have to distinguish between
the two cases.

\vskip .1in

In sections 7 and 8 we will define two further modifications
of the $2$-groupoid 
$\mathcal{B}$.
This will explain the passage to Jones type skein relations and
skein relations for framed oriented links.

\vskip .1in
 
The skein groupoid above is the result of 
applying a homology functor to a $2$-groupoid defined by chain groups
in $\textrm{imm}$, see \cite{CS1} for a set-up of homology in this framework.
It should be pointed out that this preprint of Chas and Sullivan never was published.
The reason has been that transversality on the chain level could not be established.
Many later approaches to define string topology on the chain level but in the bordism setting. 
In the low dimensional range we need for our $2$-categorical setting bordism and homology coincide.
But a more general $n$-category setting could not work on the level of homology and instead use bordism or other modifications, see
\cite{Ch} and \cite{Me}.
The chain groups \textit{linearize} the deformation $2$-groupoid constructed 
at the beginning of this section. 
It is also interesting to study skein theory from this viewpoint..
Then skein modules are the $0$-dimensional homology modules of a chain
complex with $\mathcal{R}$-coefficients and suitably $\sigma $-deformed boundary 
operator. The author intens to study this in the future. 

\section{Link theory interpretation of homology exact sequences}

We consider the exact homology sequence of the pair 
$(\textrm{imm},\textrm{emb})$ (remember the grading convention from section 3):
\[
\begin{CD}
H_1(\text{emb})\rightarrow H_1(\text{imm})\rightarrow H_1(\text{imm},\text{emb})
\rightarrow H_0(\text{emb})\rightarrow H_0(\text{imm})
\end{CD}
\]

\noindent with $i_*: H_0(\rm{emb}\rightarrow H_0(\rm{imm})$ 
surjective by transversality. 

\vskip .1in

We want to describe the geometric 
meaning of the groups and homomorphisms of the above exact
sequence. 
Obviously
$$H_0(\textrm{emb})\cong \mathbb{Z}\mathcal{L},$$
Note that $H_0(\textrm{imm}[0])\cong H_0(\textrm{map})$, where $map$ is the space of all smooth
maps in $M$. This group is isomorphic to the free abelian group on the
set of monomials
in the set $\hat{\pi}$ of free homotopy classes of loops in $M$. So
$$H_0(\textrm{imm}[0])\cong S\mathbb{Z}\hat{\pi},$$
where $S$ denotes the symmetric algebra.
Using the above isomorphisms, the homomorphism $i_*$ corresponds
to the map $\mathfrak{h}$ defined by assigning to each oriented
singular link 
its homotopy class. 

The description of the relative homology group is more interesting.

\begin{proposition} 
There is a natural isomorphism
$$H_1(\rm{imm},\rm{emb})\cong \mathbb{Z}\mathcal{L}[+1]/\mathcal{D}$$
for subgroups $\mathcal{D}[k]\subset \mathcal{L}[k+1]$, $k\geq 0$. 
For $k=0$ the subgroup $\mathcal{D}[k]$ is generated by all elements
$$K_{*+}-K_{*-}-K_{+*}+K_{-*},$$
for all $K_{**}\in \mathcal{L}[k+2]$
(differentiability relations). 
For $k\geq 1$ the subgroup
$\mathcal{D}[k+1]$ additionally has generators corresponding to
all geometric $4T$-relations and tangency relations (Theorem 3.2,
(ii) and (iii)).
\end{proposition} 

\begin{proof}
It follows from Lemma 3.2 that for $k\geq 0$ the homomorphism
$$H_1(\textrm{imm}[k,1]\cup \textrm{emb}[k],\textrm{emb}[k])\rightarrow H_1(\textrm{imm}[k],\textrm{emb}[k])$$
induced by the inclusion 
$$\textrm{imm}[k,1]\cup \textrm{emb}[k] \subset \textrm{imm}[k]$$  
is surjective.
Thus we can represent each homology class by a chain, which
maps into $\textrm{imm}[k,1]\cup \textrm{emb}[k]$ with the boundary mapping into 
$\textrm{emb}[k]$. It can be assumed that the mappings of $1$-simplices are
transverse in the sense of section 3.
Thus all parameters values of the intervals map into $\textrm{emb}[k]$, except
for a finite number mapping into $\textrm{imm}[k,1]$. The orientation of the
$1$-simplex and the usual coorientation of $\textrm{imm}[k,1]$ in $\textrm{imm}[k]$ define
a sign for each \textit{singular} parameter. This defines an integral
linear combination of elements of $\textrm{imm}[k,1]$ and thus an element in
$H_0(\textrm{imm}[k,1])$.
Next consider a relative boundary. This is given by a $2$-chain with 
boundary mapping into $\textrm{emb}[k]$. Now we apply 
Lemma 3.1, see also Remark 3.2(ii). Thus we can perturb each mapping 
of an oriented surface
$F$ into $\textrm{imm}[k]$ relative to the boundary such the set of parameters in $F$  
mapping into $\textrm{imm}[k] \setminus \textrm{emb}[k]$ is a $1$-complex, which is properly
embedded in $F$ with vertices of valence $4$ or $2$ in the interior
mapping into $\textrm{imm}[k,2]$. The contribution of a boundary element thus is
a sum of \textit{monodromies} around elements of $\textrm{imm}[k,2]$. These
elements generate the subgroup $\mathcal{D}[k+2]$ and are computed by
abelianizing the relations described in Theorem 3.2. This proves the claim. 
\end{proof}    

We summarize the discussion in the following theorem.

\begin{theorem}
The homology exact sequence of the pair $(\rm{imm},\rm{emb})$ can be 
identified with the following exact sequence:
$$
H_1(\rm{emb})\rightarrow \it{H}_1(\rm{imm})\buildrel \mu
\over \longrightarrow \mathbb{Z}\mathcal{L}[+1]
/\mathcal{D} \buildrel \partial \over \longrightarrow
\mathbb{Z}\mathcal{L}\buildrel \mathfrak{h} \over \longrightarrow
\mathbb{Z}\mathfrak{b}\rightarrow 0
$$
The homomorphism $\mu [k]$ is defined by the signed sum of all
terms in $\mathcal{L}[k+1]$ along transversal paths in 
$imm[k,1]\cup emb[k]$.  
The homomorphism $\partial [k]$ is the Vassiliev resolution
of the last double-point 
$$\mathcal{L}[k+1]\rightarrow \mathbb{Z}\mathcal{L}[k].$$
\end{theorem}

\begin{proof} The result follows easily from proposition 4.1
and its proof.
\end{proof} 

The above sequence provides a description of the kernel of the
$\mathfrak{h}$. This is the forgetful homomorphism from isotopy to homotopy,
mapping \textit{quantum
obserbales} to their semi-classical limits. 
  
\begin{corollary} There is the isomorphism 
$$
\mathbb{Z}\mathcal{L}\cong \mathbb{Z}\mathfrak{b}\oplus
coker(\mu ).
$$ 
\end{corollary}

The result of the corollary is just a different description of
the $2$-groupoid structure discussed in section 3.
It shows in a neater way the distinction in terms of \textit{local}
relations and string topology homomorphism.

\vskip .1in

Recall that we 
have special \textit{geometric} splitting homomorphisms
$$\mathfrak{s}: \mathbb{Z}\mathfrak{b}\rightarrow 
\mathbb{Z}\mathcal{L}.$$
These are defined by realizing chord diagrams in $M$ by  
corresponding immersions.
Then for $k=0$, a sequence of free homotopy classes is mapped to a link
with components realizing those free homotopy classes.
Recall that $\mathfrak{b}[0]$
is the set of monomials in the set $\hat{\pi }$, and
$$\mathfrak{s}[0]: \mathfrak{b}\rightarrow \mathcal{L}[0].$$
Let
$\mathfrak{s}[0](\alpha )=:K_{\alpha }\in \mathcal{L}[0]$
be the standard link
corresponding to the monomial $\alpha \in \mathfrak{b}$.

\vskip .1in

In the following let $\sigma $ denote either a skein potential or the
Vassiliev potential $\sigma_V$.

\vskip .1in

\begin{proof} \textbf{[of Theorem 1.1 and Theorem 1.5]} (in the Conway boundary case)
The claims follow 
from Theorems 2.1, 2.2, 2.3 and 3.1 
applied to the categories $\mathcal{B}[0]$ respectively
$\mathcal{B}$ (or $\mathcal{C}[0]$ respectively $\mathcal{C}$), see the
discussion at the end of section 3. The models are
$\mathfrak{b}[0]$ respectively $\mathfrak{b}$.
The exact sequence of Theorem 4.1 (and the splitting in Corollary 4.1) 
describe in a systematic way \textit{all}
homology classes of $1$-cycles with boundary 
$K-\mathfrak{s}\circ \mathfrak{h}(K)$ for $K\in \mathcal{L}$.
The additional information we have here is that the relations coming from
expansion of \textit{closed} $1$-morphism up to $2$-morphisms, \textit{factor}
through the homomorphism $\mu $. 
In the inductive argument we actually have to consider the lift of
$\widetilde{\sigma} \circ \mu$ to $R\mathcal{L}$. 
Here we have defined:
$$\widetilde{\sigma}: \mathbb{Z}\mathcal{L}[+1]/\mathcal{D}\rightarrow
\mathcal{R}\mathcal{L}/\langle \sigma (\mathcal{D})\rangle,$$
where $\langle \sigma (\mathcal{D})\rangle $ is the submodule
generated by the subgroup $\sigma (\mathcal{D})$.
\end{proof}

\begin{remark} 
Let $\iota :\mathcal{R}\mathcal{L}[0]\rightarrow
\mathcal{S}(\sigma )$ be the projection. 
Then $\iota \circ \widetilde{\sigma} \circ \mu =0$. 
Here we use that $\iota $ factors through $\mathcal{R}/\langle
\mathcal{D} \rangle$ and in $\mathcal{S}(\sigma )$ the following
relation holds:
$$\sigma (K_{*+}-K_{*-}-K_{+*}-K_{-*})=
\partial (K_{*+}-K_{*-}-K_{+*}-K_{-*})=
K_{++}-K_{-+}....$$  
But for $K_+\neq K_-$
a skein relation of the form $K_+-K_--\sigma (K_*)$ is not contained
in the image of $\widetilde{\sigma }\circ \mu $, even if we consider
the extended homomorphism 
$$\widetilde{\sigma}\circ \mu : 
H_1(\textrm{imm})\otimes \mathcal{R}\rightarrow \mathcal{R}\mathcal{L}/\langle
\sigma (\mathcal{D})\rangle.$$
In fact, there is a more natural way of associating a string topology
homomorphism to a given skein relation or Vassiliev relation. This
will be discussed in section 10.
\end{remark}  

\section{Some general results about skein potentials and skein modules}

We work in the graded $2$-groupoid $\mathcal{B}$ 
or in the $2$-groupoid $\mathcal{B}[0]$. Then
$\rm{hom}[k]$ is naturally identified with $\mathbb{Z}\mathcal{L}[k+1]$.
Because of the functoriality property in the definition, a skein
potential is determined by its values on the generators $\mathcal{L}[k+1]$.
Thus we consider a skein potential as a map
$\mathcal{L}[+1]\rightarrow \mathcal{L}$ using the usual graded
notation.

\vskip .1in

For some of the following results
we have to restrict the choice of geometric models.
Let
$$\sigma : \mathcal{L}[1]\rightarrow \mathcal{R}\mathcal{L}[0]$$
be a skein potential with respect to the ideal $I\subset \mathcal{R}$.
Let $K_{\alpha }:=\mathfrak{s}(\alpha )$ be the \textit{standard link}
with homotopy class $\alpha $.

\begin{definition}
A choice of geometric models 
$$\mathfrak{s}: \mathfrak{b}[0]\rightarrow \mathcal{L}[0]$$
is called \textit{nice} if the
following two conditions hold: (i) for free homotopy classes which are
multiple times contained in a monomial $\alpha \in \mathfrak{b}[0]$
there exists a self-isotopy of the standard link $K_{\alpha }$,
which arbitrarily changes the order of the corresponding components.
(ii) for each trivial free homotopy class
in $\alpha $, the standard link $K_{\alpha }$
contains some unlinked and unknotted circles
in a $3$-ball separated from the rest of the link.
\end{definition} 

\begin{lemma}
It is always possible to choose nice geometric models.
\end{lemma}

\begin{proof}
If $\alpha $ contains a free homotopy class $\beta \in \hat{\pi}$
with multiplicity $n_{\beta }$ then we can choose $n_{\beta }$
parallel copies in the link $K_{\alpha}$. Of course trivial homotopy
classes can be represented by unlinked and unknotted components 
separated from the link.
\end{proof}

\begin{remark} Let $M$ be the solid torus $S^1\times D^2$, 
or a Lens space $L(p,q)$.
Then $L(p,q)$ contains a solid torus and each link is isotopic into
this torus.
Then we can choose nice standard links contained in $S^1\times D^2$, which
are descending. In fact, there is an isotopy of the solid torus which
changes the order of components with multiple free homotopy classes of
components.  
\end{remark}
 
\begin{definition} A skein potential 
$\sigma : \mathcal{L}[+1]\rightarrow \mathcal{R}\mathcal{L}$
is called
\textit{local} if it is of the form
$$\sigma (K_*)=\sum_{i=1}^n{h_i}K_i,$$
where for $1\leq i\leq n$, $a_i\in \mathcal{R}$ and $K_i$ are singular
links defined by
replacing the two intersecting arcs in the oriented $3$-ball centered 
about the last 
double-point $*$ by the $2$-tangles $t_i$.
We assume that $a_i\in I$ if the number of components of $K_i$ is not
smaller than the number of components of $K_*$.
\end{definition}

The following is obvious.

\begin{lemma} Each local skein potential 
$$\sigma : \mathcal{L}[1]\rightarrow \mathcal{L}[0]$$ 
extends to a local skein potential
$$\mathcal{L}[+1]\rightarrow \mathcal{L}.$$
In fact for each $1\leq i\leq k+1$ there exist maps
$$\sigma_i : \mathcal{L}[k+1]\rightarrow \mathcal{L}[k]$$
defined by applying $\sigma $ to the $i$-th double-point
of a singular link. \hfill{$\square$}
\end{lemma}

Similarly we can define Jones and Conway boundaries on 
$\mathcal{L}[k+1]$ for $1\leq i\leq k+1$ 
by applying $\partial$ to the $i$-th double point.
Note that for $i<j$

$$\partial_i\circ \sigma_j= \sigma_{j-1}\circ \partial_i.$$

\vskip .1in

Recall that a skein relation is insensitive (with respect to 
differentiability relations for $k\geq 1$ if for each 
differentiability element $d\in \mathcal{L}[k+1]$ we can \textit{choose} 
a preimage $\delta^{-1}(\sigma (d))$ such that 
$\sigma (\delta^{-1}(\sigma (d))=0$.

\begin{remark}
Global insensitivity for the skein potentials of the graded category
is rare because of the $4T$-relations. In fact, 
the usual Vassiliev potential $\sigma_V$ is \textit{not} robust.
Also the extensions of skein relations 
according to Lemma 5.1 are not robust. 
In particular the following result only applies to the usual skein
potentials. 
\end{remark}

\begin{proposition} Each local skein potential 
$\mathcal{L}[1]\rightarrow \mathcal{R}\mathcal{L}[0]$ is robust.
Also, each local skein potential
$\mathcal{L}[+1]\rightarrow \mathcal{R}\mathcal{L}$ is robust
with respect to the differentiablity morphisms.
\end{proposition}

\begin{proof}
Each differentiability element $d$ is in the image of 
$$\partial_{k+2}-\partial_{k+1}: 
\mathcal{L}[k+2]\rightarrow \mathbb{Z}\mathcal{L}[k],$$
which maps $K_{**}\in \mathcal{L}[k+2]$ to
$$K_{*+}-K_{*-}-K_{+*}-K_{-*}.$$
We have to show that for each $y\in \mathcal{L}[k+2]$ we can
choose a preimage, denoted $\partial ^{-1}(y)$, such that
$$\sigma \circ \partial^{-1} \circ \sigma \circ \partial_{k+2}(y)=
\sigma \circ \partial^{-1}\circ \sigma \circ \partial_{k+1}(y).$$
Recall that $\sigma $ without index operates on the last double-point.
Of course we can choose the preimage on the left hand side of the
equation such that $\partial^{-1} \circ \sigma
\circ \partial_{k+2}(y)=\sigma_{k+1}(y)$ and similarly on the right hand
side $\partial^{-1}\circ \sigma \circ \partial_{k+1}(y)=\sigma (y)$
This follows from locality, e.\ g.\ in the first case: 
$\sigma \circ \partial_{k+2}=
\partial \circ \sigma_{k+1}$.
Then the claim follows using locality: 
$$\sigma \circ \sigma_{k+1}=\sigma \circ \sigma_{k+2}.$$
\end{proof}

In the following it will sometimes turn out to be useful to describe skein
modules $\mathcal{S}(\sigma )$ as modules over the corresponding 
skein module $\mathcal{S}(D^3)$.  

\begin{proposition} For each local skein potential $\sigma $ the module 
$\mathcal{S}(M;\sigma )$ is a module over the commutative 
skein algebra $\mathcal{S}(D^3;\sigma )$,
and a module over the skein algebra $\mathcal{S}(\partial M;\sigma )$.
\end{proposition}

\begin{proof} The disjoint union with links in a $3$-ball
separated from a given link is a well-defined operation.
Locality of the skein potential implies that this defines a 
pairing
$$\mathcal{S}(D^3;\sigma) \otimes \mathcal{S}(M;\sigma )\rightarrow
\mathcal{S}(M;\sigma ).$$
The same holds for $\mathcal{S}(\partial M;\sigma )$ using a collar
of $\partial M$. In fact, the case of $D^3$ is a special case where 
we replace $M$ by the complement of some open $3$-ball. Because of
locality this does not change the skein module. Also this allows to
identify the skein modules of $D^3$ and $S^2\times [0,1]$.
The algebra structures on $\mathcal{S}(M;\sigma )$ is defined
in the usual way using the $[0,1]$-structure. 
\end{proof}

If $\partial M$ is not a union of tori then the skein algebra usually
is not commutative. 

\vskip .1in

Let $\sigma $ be the \textit{Conway skein potential} $\sigma (K_*)=hK_0$
respectively $zK_0$ for a double-point of the same respectively
different components. In this case we add the vaccuum relation
$(q^{-1}-q)\emptyset =hU$. Note that $hU\neq \sigma (K_*)$ for any 
singular link $K_*\in \mathcal{L}[1]$. 
Recall that by definition our category contains 
the empty link and thus the empty $1$-morphism in $\mathcal{L}[1]$, which 
is the identity morphism of 
the empty link. 
When working with vaccum relations we actually 
only replace the trivial skein
relation $\sigma (\emptyset )=(q^{-1}-q)\emptyset$ 
e.\ g.\ in the Jones case (such that $(\partial -\sigma )(\emptyset )=0$
automatically holds) by the relation above.
This will have the advantage that
the kink relations (see section 6) can be already included into the 
structure of the ring $\mathfrak{R}:=\mathcal{S}(D^3;\sigma )$.

\section{The homomorphism $\mu $ and the topology of $M$}

For $k\geq 1$ let $\mathcal{K}[k] \subset \mathcal{L}[k]$ denote the set of
$k$-immersions with a distinguished self-crossing denoted $*$ 
such that one of the lobes of the
singular component of $K_*$ bounds a disk intersecting $K_*$ only in
its boundary along the lobe. We call this a \textit{kink} $k$-immersion.
There is a unique homology class $\gamma (K_*)\in H_1(\textrm{imm}[k-1])$
represented by the transversal path in $\textrm{imm}[k-1]$, which is defined by 
running through the natural
isotopy from $K_+$ to $K_-$ and the crossing change at $K_*$.
The disk bounding the trivial component in the smoothing $K_0$ defines
a null-homology of this loop in $\textrm{map}[k-1]$. 
Here $\textrm{map}[k]$ is the space of all
smooth maps of circles in $M$ with exactly $k$ fixed double points, with definition analogous 
to the definition of $\textrm{imm}[k]$ is section 3.
This null-homology
contains a single point contained in $\textrm{map}[k-1]\setminus \textrm{imm}[k-1]$, 
where the corresponding map contains the mapping of a
circle which embeds except at a single point with 
vanishing tangent vector.
The corresponding element is non-trivial in $H_1(\textrm{imm}[k-1])$. 
Note that
$$(\mu [k-1])(\gamma (K_*))= K_*,$$
and
$$(\partial [k-1])(K_*)=K_+-K_-=0$$
by construction for each $K_*\in \mathcal{K}[k]$.

\begin{remark} 
Let $\widetilde{\mathcal{L}}[k]$ denote the set of isotopy classes of
ordered singular $k$-links (the set of path components of the space
$\widetilde{\textrm{emb}}[k]$ of the space of ordered 
$k$-embeddings).
Then there is a well-defined \textit{onto} map 
$$
\chi [k]: \widetilde{\mathcal{L}}[k-1]\rightarrow
\mathcal{K}[k],
$$
which assigns to each isotopy class of ordered $(k-1)$-singular link the 
isotopy class of the immersion with an additional
kink in the 
\textit{first}
component away from possible preimages of double-points. 
\end{remark}

\begin{lemma}
The kernel of the epimorphism
$$j_*: H_1({\rm imm})\rightarrow H_1(\rm{map})$$
is generated by all elements of the form
$\gamma (K_*)$ for all $K_*\in \mathcal{K}[+1]$.
In particular the linear extension of
$\chi $ maps $\mathbb{Z}\widetilde{\mathcal{L}}$ onto 
$ker(j_*)$.
\end{lemma}

\begin{proof}
This is another application of Lin transversality, 
respectively a modification of Lemma 3.1, see also Remark 3.2 (i). 
Consider a mapping of
a surface $F$ into $\textrm{map}[k]$, which is transversal along $\partial F$. 
It can be
approximated relative $\partial F$ to the following way: The set
of points in $F$, which map into $\textrm{map}[k]\setminus \textrm{emb}[k]$ consists of a 
$1$-complex embedded in $F$, with vertices of possible orders $2, 4$ or
$1$ in the interior. Those of order $4$ or $2$ are mapped to
$\textrm{immk}[k,2]$. 
Those of order $1$ are mapped to a smooth map,
which is an immersion with $k$ double points 
and a single point in the complement of the double points
with vanishing tangent. The point with vanishing tangent appears in
the boundary of $\textrm{imm}[k,1]\subset \textrm{imm}[k]$.  
Now cut out small disks 
from $F$ around the vertices of order $1$. The restriction to the 
the boundary of a disk
represents an element in $H_1(\textrm{imm}[k])$ of the form $\gamma (K_*)$
for $K_*\in \mathcal{K}[k]$.  
\end{proof}

Let in the presentation $K_{**}\in
\mathcal{K}[2]$
the first place indicate a self-crossing with a bounding lobe.
Then $K_{*+}, K_{*-}$ are contained in
$\mathcal{K}[1]$.
This is not necessarily the case for $K_{+*}, K_{-*}$ but 
$K_{+*}=K_{-*}$.
Thus we have the following commuting diagram with exact rows, 
where $\mathcal{D}_{\bullet}$  is the subgroup of
$\mathbb{Z}(\mathcal{L}[1]\setminus \mathcal{K}[1])$, 
which is generated by the 
differentiability relations with 
\textit{all} four terms in $\mathcal{L}[1]\setminus \mathcal{K}[1]$:  
$$
\begin{CD}
0@>>>\mathcal{D}\cap \mathbb{Z}\mathcal{K}[1]@>>>\mathcal{D}[0]@>>>
\mathcal{D}_{\bullet }@>>>0
\\
@. @V\subset VV @V\subset VV @V\subset VV @. \\
0@>>>\mathbb{Z}\mathcal{K}[1]@>>>\mathbb{Z}\mathcal{L}[1]@>>>
\mathbb{Z}(\mathcal{L}[1]\setminus \mathcal{K}[1])@>>>0
\end{CD}
$$
and the diagram of homomorphisms
$$
\begin{CD}
H_1(\text{imm}[0])@>\mu [0]>> \mathbb{Z}\mathcal{L}[1]/\mathcal{D}[0] \\
@Vi_*[0]VV @VVV \\
H_1(\text{map}[0])@>\mu_{\bullet}>>\mathbb{Z}(\mathcal{L}[1]\setminus \mathcal{K}[1])/\mathcal{D}_{\bullet}
\end{CD}
$$
with the vertical right map defined by projection.

\begin{corollary}
There is the induced boundary homomorphism
$\partial_{\bullet }$ and the exact sequence
$$
\begin{CD}
H_1(\rm{emb}[0]) \rightarrow \it{H}_1(\rm{map}[0])@>\mu_{\bullet }>>
\mathbb{Z}(\mathcal{L}[1]\setminus
\mathcal{K}[1])/\mathcal{D}_{\bullet}
 @>\partial_{\bullet }>> 
\mathbb{Z}\mathcal{L}[0]@>\mathfrak{h}>> S\mathbb{Z}\hat{\pi }. 
\end{CD}
$$

\vskip .1in

There is the induced isomorphism
$$\rm{coker}(\mu [0])\cong \rm{coker}(\mu_{\bullet }).$$
\end{corollary}

\vskip .1in

\begin{proof} Most of the claims follow by the construction of
the homomorphisms. The exactness at $H_1(\textrm{imm}[0])$ follows because
$\mu [0]$ restricts to an epimorphism
$$\mathcal{K}[0]\rightarrow
\mathbb{Z}\mathcal{K}[1]/(\mathcal{D}[1]\cap \mathbb{Z}\mathcal{K}[1]).$$
\end{proof}

The main link theoretic consequences will follow from
Theorem 1.1.

\begin{definition}
Let $M$ be a compact $3$-manifold. $M$ is called \textit{atoroidal}
(respectively \textit{aspherical}) if each 
$\pi_1$-injective map of a torus (respectively map of a $2$-sphere)
in $M$ is
homotopic into $\partial M$  
\end{definition}

Each irreducible $3$-manifold is aspherical.
Note that each hyperbolic $3$-manifold is aspherical and atoroidal.

\begin{theorem}
Suppose $M$ is aspherical and atoroidal. Then $\mu_{\bullet }=0$. 
\end{theorem}

\begin{proof}
Most of the arguments are already contained in \cite{K1} and \cite{K2}. 
Fix a component $\textrm{map}_{\alpha }$ of the space $\textrm{map}[0]$ 
corresponding to $\alpha \in
\mathfrak{b}[0]$. 
Note that
$$H_*(\textrm{map}[0])\cong \bigoplus_{\alpha \in \mathfrak{b}[0]}H_*(\textrm{map}_{\alpha
}).$$
The Hurewicz homomorphism
$$\pi_1(\textrm{map}[0],f_{\alpha })\rightarrow H_1(\textrm{map}_{\alpha })$$
is onto.
Let $\widetilde{\textrm{map}}$ be the space of ordered smooth maps with the
fat diagonal excluded. Let $\widetilde{\textrm{map}}_{\alpha }$ denote 
the preimage of $map_{\alpha }$ under the covering projection
$$\widetilde{\textrm{map}}\rightarrow \textrm{map}.$$
Note that the path components of $\widetilde{\textrm{map}}_{\alpha }$ are
labelled by the orderings of $\alpha $. 
Let $a$ be an ordered sequence of elements in $\hat{\pi }$ corresponding to
$\alpha $.  We choose a representative embedding 
$\widetilde{f}_a \in \widetilde{\textrm{map}_a}$. Note that $f_a$ is ordered.
If $a$ does not contain multiple elements of $\hat{\pi }$ then the 
component $\widetilde{\textrm{map}}_a$ is homeomorphic to $\textrm{map}_{\alpha }$.
But multiple occurences of elements of $\hat{\pi }$ in $a$ imply that
there are homotopies joining ordered links with the same underyling
unordered link. Then the covering of the component is non-trivial
and the injection
$$\pi_1(\widetilde{\textrm{map}}_a,f_a)\rightarrow \pi_1(\textrm{map}[0],f_{\alpha })$$
is not necessarily onto. This means that a loop in $\textrm{imm}_{\alpha }$
does not necessarily lift to a loop in $\textrm{imm}_a$. 
In this case
we choose the embeddings $f_a$ in the following symmetric way: If some element
in $a$ appears multiple times then we choose
corresponding parallel components for $f_a$. Thus there exists an isotopy of 
$M$ which changes the order of components.
We can compose a given loop $\gamma $ in $\textrm{map}_{\alpha }$ with 
loops in $\textrm{emb}$ changing
the order in such a way that the composition of loops lifts to a loop
in $\textrm{map}_a$. Note that the composition still  
has the same image under $\mu_{\bullet}$ as $\gamma $. 
So the image of $\mu_{\bullet}$ on the component $H_1(\textrm{map}_{\alpha })$
corresponds to the image of a homomorphism defined on
$$\pi_1(\textrm{map}[0];f_a)\cong \prod_{i}\pi_1(\textrm{map}_{a_i},f_{a_i}).$$
Note that we can fill in the infinite codimensional fat diagonal
without changing the fundamental group and then identify our mapping space
with the product of single component mapping spaces.
This shows that the image of $\mu_{\bullet}$ is generated by the images of
elements represented by loops in $\textrm{imm}[0]$ which fix all but one
component. Let $L$ denote the union of those components which are
fixed during the homotopy. In the following note that free homotopy
implies homology in our case.

Next consider the singular torus map given by the adjoint of the
non-constant component of such a loop in $\textrm{imm}[0]$.
First assume that this map is \textit{not} essential.
Then the image of $\pi_1(S^1\times S^1)\cong \mathbb{Z}\oplus
\mathbb{Z}$ in $\pi_1(M)$ is cyclic (for details see \cite{K1} ). 
Thus on a neighbourhood of wedge $S^1\vee S^1$ in $S^1\times S^1$ 
the map can be homotoped into
the tubular neighbourhood $T$ of a knot embedded in $M$.  
The full torus map is homotopic in $T$ up to a map of a $2$-sphere.
We can easily arrange that the singular torus can be represented as
a connected sum of singular torus in $N$ and parallel copies of 
$2$-spheres contained in the collar neighbourhood of the union
of $2$-sphere components  
in $\partial M$. Now consider a singular torus contained in $T$.
We want to apply the homotopy exact sequence of the fibration (see
\cite{V2}, Appendix):
$$\Omega M \rightarrow \textrm{map}[0] \rightarrow M,$$ 
where $\Omega $ is the usual based loop functor for $M=T$ a solid
torus.
It is easy to see that all elements of $\pi_1(\textrm{map}[0];f)$ can be represented by
loops in $\textrm{imm}[0,1]$. (Each generator of $\pi_1(\textrm{map}[0];f)$ can be represented 
by longitudinal rotation of the knot $f$ in $T$.)  
The possible connected sum arcs for the $2$-sphere contribution can be
homotoped along. But since the two-spheres and connecting tubes can be
assumed embedded the resulting loop is still contained in $\textrm{imm}[0,1]$.      
Finally since the core of $T$ and connecting arcs are
$1$-dimensional and the $2$-sphere maps into a collar of the boundary, 
the intersections with $L$ can be avoided and the loop of maps 
is contained in $emb$. Thus the image under $\mu_{\bullet}$
vanishes.

It remains to discuss the case of an essential torus map. If this map is
homotopic into the boundary then the boundary contains a torus and we
can homotope into a collar neighbourhood $N\cong S^1\times S^1\times
[0,1]$ of a torus boundary component, in particular avoiding $L$.
As before we can argue in the
mapping space of $N$ and find a free homotopy of the torus map thus
inducing a homology of the given loop into $\textrm{emb}[0]$.   
\end{proof}

\begin{example} (a) Let $M=S^3$ and consider a \textit{connected sum}
element in $\mathcal{L}[1]$.
Thus $*$ is a self-crossing and $K_0$ is
the union of two nonempty links contained in disjoint balls.  
Then $K_+=K_-$ by rotation. The full $2\pi $-rotation 
defines a loop $\ell$ in $imm$.
Note that $\mu (\ell)$ is represented by $K_*$ so obviously not
trivial in $\mathcal{L}[1]$. Now consider a path in $\textrm{imm}$, which joins
one of the two link pieces in $K_0$ with trivial link by only crossing
changes. Since the rotation can be 
performed along this deformation there exists a free homotopy of the
loop $\ell $ into a loop for which the rotation is a kink rotation. 
Thus this phenomenon is not measured by $\mu_{\bullet}$.

\noindent (b) Suppose $M$ is not aspherical. Then either $M$ contains
$S^2\times S^1$ or $M$ is a connected sum. In each case it is easy to
show that $\mu_0\neq 0$, at least modulo the Poincare conjecture. For
example in the second case take a connected sum immersion with each
lobe homotopically non-trivial. A natural isotopy across the $S^2$
easily changes $K_+$ into $K_-$. But now $K_*$ cannot be trivial in
$(\mathbb{Z}\mathcal{L}[1]\setminus \mathcal{K}[1])/\mathcal{D}_{\bullet}$.
In order to show this map into $S\mathbb{Z}\hat{\pi}$ by taking
homotoppy classes.
\end{example} 

\section{The Jones deformation}

In this section we  describe the passage from the Conway 
boundary to the Jones boundary case. There is a natural way to
introduce the $q$-structure abstractly on 
the groupoid level. In the topological case this amounts to
a lift of twisted homology to the deformed fundamental groupoid.

Let $\mathcal{C}$ be a $2$-groupoid and let
$$\varepsilon : \textrm{hom} / \textrm{mor} \rightarrow \mathbb{Z}$$
be a \textit{twist} homomorphism (or abstract local system), i.\ e.\
$$\varepsilon (u\circ v)=\varepsilon (u)+\varepsilon (v)$$
for all $u,v\in hom$. (We identify elements in $\textrm{hom}$ with its
equivalence classes under the action of $\textrm{mor}$.)
It follows that
$$\varepsilon (u^{-1})=-\varepsilon (u)$$
since $\textrm{mor}(u\circ u^{-1},1)\neq \emptyset$
and $\varepsilon (1)=0$ for each identity $1$-morphism
$1\in \textrm{hom}(x,x)$.

We describe a $2$-groupoid $\mathcal{C}_q$, the \textit{Jones
deformation} of $\mathcal{C}$. 
To avoid confusion we will write the compositions in the category
$\mathcal{C}_q$ as $\diamond$. 

Let for $i=0,1,2$ 
$$(\textrm{hom}_i)_q:=\{q^jw|w\in \textrm{hom}_i,j\in \mathbb{Z} \},$$
where we use the notation $q^jw:=(j,w)$, i.\ e.\ 
$(\textrm{hom}_i)_q$ is the subset of scalar multiples 
of basis elements in free abelian group on $\textrm{hom}_i$.
To avoid confusion we write $q^0u$ for the image of
$u\in \textrm{hom}_i$ in $(\textrm{hom}_i)_q$ for all $i$.

\vskip .1in

There is the natural action of $\mathbb{Z}$ on $(\textrm{hom}_i)_q$
defined by 
$q^kw:=q^{j+k}u\in (\textrm{hom}_i)_q$ for $w=q^ju$, $u\in \textrm{hom}_i$ and 
$j,k\in \mathbb{Z}$ and $i=0,1,2$.

\vskip .1in

We define the target and source maps and the compositions.
If $u\in \textrm{hom}(x,y)$ then let
$q^0u\in \textrm{hom}_q(x,q^{-2\varepsilon(u)}y)$
and $q^iu\in hom_q(q^ix,q^{i-2\varepsilon (u)y})$
for $i\in \mathbb{Z}$.
Thus 
$$\textrm{sour}_q(q^iu)=q^i\textrm{sour}(u)\in \mathfrak{ob}_q$$
and
$$\textrm{targ}_q(q^iu)=q^{i-2\varepsilon (u)}\textrm{targ}(u)\in \mathfrak{ob}_q$$
for each $u\in \textrm{hom}$.
The source and target maps are \textit{equivariant}
with respect to the $\mathbb{Z}$-action.

Let $v\in \textrm{hom}(x,y)$ and $u\in \textrm{hom}(y,z)$.
It follows from the \textit{homomorphism} property of $\varepsilon$
that, if $u\circ v$ is defined then 
$$q^{-2\varepsilon (v)}u\diamond q^0v:=q^0(u\circ v)$$ is 
defined, and is a $1$-morphism
from $x$ to $q^{-2\varepsilon (u)-2\varepsilon (v)}y=
q^{-2\varepsilon (u\circ v)}y$ in $\mathcal{C}_q$.
A composition with $q^iv$ is defined such that equivariance holds:
$$q^i(a\diamond b)=(q^ia)\diamond (q^ib)$$
for $a,b\in \textrm{hom}_q$.
Note that $\textrm{hom}_q(q^ix,q^jy)=\emptyset$ if $j-i$ is odd.
Thus $\mathfrak{ob}_q$ naturally splits into a disjoint union of two 
sets with homomorphisms only between objects in each of the two sets.
 
Finally we define horizontal and vertical composition
of $2$-morphisms. Note that if $\textrm{mor}(u,v)\neq \emptyset $ then
$\varepsilon (u)=\varepsilon (v)$. Thus after identification of
$u,v\in \textrm{hom}$ with the corresponding elements in $\textrm{hom}_q$ (note that
the targets have changed) we can identify each $h\in \textrm{mor}(u,v)$ with
the corresponding $2$-morphism in $\textrm{mor}_q$. We will have 
$q^ih\in \textrm{mor}_q(q^iu,q^iv)$ for $h\in \textrm{mor}(u,v)$. The composition
in $\textrm{mor}$ induces obvious compositions $\diamond_1, \diamond_2$ in $\textrm{mor}_q$.

\vskip 0.1in

\begin{theorem}
\noindent \textup{(i)} $\mathcal{C}_q=(\mathfrak{ob}_q,\rm{hom}_q,\rm{mor}_q)$ defines a $2$-groupoid.

\noindent \textup{(ii)} There is an obvious functor of $2$-groupoids 
$p_q: \mathcal{C}_q\rightarrow \mathcal{C}$
defined by $q\mapsto 1$, which is onto on all sets $\rm{hom}_j$
for $j=0,1,2$.

\noindent \textup{(iii)} There is a natural $\varepsilon $-twisted functor 
$i_q: \mathcal{C}\rightarrow \mathcal{C}_q,$
where $i_q(a):=q^0a$ for $a\in \rm{hom}_i$.
This means that for $u,v\in \rm{hom}$:
$$i_q(u\circ v)=q^{-2\varepsilon (v)}u\diamond q^0(v).$$
For $v_1,v_2\in \rm{hom}(\it{x},\it{y})$, $u_1,u_2\in \rm{hom}(\it{y},\it{z})$ and
$g\in \rm{mor}(\it{u}_{\rm{1}},\it{u}_{\rm{2}})$, $h\in \rm{mor}(\it{v}_{\rm{1}},\it{v}_{\rm{2}})$ we have
$$i_q(g\circ _2h)=(q^{-2\varepsilon (v_1)}g)\diamond_2 q^0h.$$
For $h_1\in \rm{mor}(\it{v}_{\rm{2}},\it{v}_{\rm{3}})$ and $h_1\in \rm{mor}(\it{v}_{\rm{1}},\it{v}_{\rm{2}})$
with $v_i\in \rm{hom}(\it{x},\it{y})$ for $i=1,2,3$ we have the untwisted identity
$$i_q(h_1\circ_1 h_2)=(q^0h_1)\diamond_1 (q^0h_2).$$  

\noindent \textup{(iv)} The equation of functors holds: $p_q\circ i_q=id$.
\end{theorem}

\begin{proof}
Let $x,y\in \mathfrak{ob}$.
and $u\in \textrm{hom}(x,y)$ with uniquely determined 
$u^{-1}\in \textrm{hom}(y,x)$. Then 
$(q^0u)^{-1}:=q^{-2\varepsilon (u)}u^{-1}\in \textrm{hom}_q(q^{-2\varepsilon (u)}y,x)$. 
We are using that $\varepsilon (u^{-1})=-\varepsilon (u)$.
The inverse in general is defined by equivariance
$$(q^iu)^{-1}:=q^i(q^0u)^{-1}.$$
for $i\in \mathbb{Z}$.
The remaining arguments are straightforward and left to the reader.
\end{proof}

\begin{remark} 
In terms of graphs 
(see \cite{BD} for the graphical explanation of categories) this means
that we take the original graph corresponding to the category
$\mathcal{C}$ and take $\mathbb{Z}\times \mathfrak{ob}$. Then define
the edges according to the above procedure determined by the twist.
This is like lifting 
paths into a covering space corresponding to $\varepsilon $. 
Note that the twist is a local system on the graph associated to the
$2$-groupoid.
\end{remark}

\begin{theorem}
Let $\mathcal{R}\supset \mathbb{Z}$ be a commutative ring with $1$.
Let $\mathcal{C}$ be a $2$-groupoid with linear potential
$\mathfrak{a}: \rm{hom} \rightarrow \it{\mathcal{R}\mathfrak{ob}}$.
Suppose that $\rm{hom} =\it{F(S,\circ)}$ or $\rm{hom}=\it{A(S,\circ )}$ for a subset 
$S\subset \rm{hom}$.
There exists the skein potential in $\mathcal{R}[q^{\pm 1}]$:
$$\mathfrak{a}_q: \rm{hom}_{\it{q}}\rightarrow \it{\mathcal{R}}\mathfrak{ob}_q\cong
\mathcal{R}[\it{q}^{\pm \rm{1}}]\mathfrak{ob},$$
defined uniquely by functoriality and $\mathbb{Z}$-equivariance:
$$\mathfrak{a}_q(q^0s):=\mathfrak{a}(s)$$
for all $s\in S$
\end{theorem}

By assumption the following equation holds:
$$\mathfrak{a}_q((q^0s)^{-1})=\mathfrak{a}(s^{-1})=
-q^n\mathfrak{a}(s).$$
The next useful formula is easily proved from the
definitions.

\begin{proposition} Let $\mathfrak{a},S$ be as in the last theorem
and and let the $2$-groupoid $\mathcal{C}$ be the the homotopy Vassiliev groupoid $\mathcal{C}(M)$
or the homology Vassiliev groupoid $\mathcal{B}(M)$.
Let $s_j\in S$ and $\epsilon_j\in \{\pm 1\}$ for $j=1,\ldots ,r$. Then 
the following formula holds:
$$\mathfrak{a}_q(i_q(s_r^{\epsilon _r} \circ
s_{r-1}^{\epsilon_{r-1}}\circ \ldots \circ s_1^{\epsilon_1}))=
\sum_{j=1}^r\epsilon_jq^{-\epsilon_j\varepsilon (s_j)}\mathfrak{a}(s_j)$$
\end{proposition}

\begin{remark} 
We have for $\mathcal{C}(M)$ respectively $\mathcal{B}(M)$, 
$\textrm{hom}[k]=F(\mathcal{L}[k+1],\circ )$ 
respectively $\textrm{hom}[k]=A(\mathcal{L}[k+1],\circ )$.
We define $\varepsilon (K^{\pm})=\pm 1$ for $K\in \mathcal{L}[k+1]$.
Note that 
$K_*$ corresponds to the morphism $s_*$ defined by the crossing change
from $-$ to $+$ at $*$.
We have $\textrm{hom}_q[k]=\textrm{hom}[k]_k\subset \mathbb{Z}[q^{\pm }\mathcal{L}[k+1]$
in a similar way (see also the description of
$\mathbb{Z}[q^{\pm}]\mathcal{L}[k+1]$ as twisted homology below).
The following is important. In the $q$-deformed case $K_*$ corresponds
to
the morphism defined by $K_-\mapsto q^{-2}K_+$ and
$K_*^{-1}$ corresponds to the morphism
$q^{-2}K_+\mapsto K_-$, which the $q^{-2}$-multiple of the usual
morphism $K_+\mapsto q^2K_-$. It is this correspondence which makes
$\textrm{hom}_q\subset \mathbb{Z}[q^{\pm}]\mathcal{L}[+1]$ a morphism.
\end{remark}

Let $\sigma $ be a skein potential in $\mathcal{R}$.
By Theorem 7.2 $\sigma $ induces a skein potential
$\mathcal{L}[+1]\rightarrow \mathcal{R}\mathcal{L}$ in the following
way. Note that $\sigma $ induces a linear potential in the usual way.
This linear potential induces a linear potential in the deformed 
category $\mathcal{B}_q$. 
This defines a skein potential 
$$\sigma_q: \mathcal{L}[+1]\rightarrow \mathcal{R}[q^{\pm}]\mathcal{L}.$$

Then
$$\sigma_q(i_q(s_r^{\epsilon _r} \circ
s_{r-1}^{\epsilon_{r-1}}\circ \ldots \circ s_1^{\epsilon_1}))=
\sum_{j=1}^r\epsilon_jq^{-\epsilon_j}\sigma(s_j)$$
where $s_j=K_{*,j}$ is a composable sequence of elements in
$\mathcal{L}[k+1]$.  

\begin{definition} We define that $(\mathcal{C},\varepsilon )$ has \textit{kinks}
if for each $x\in \mathfrak{ob}$ there exists a kink morphism $k_x\in \textrm{hom}(x,x)$
such that $\varepsilon (k_x)=1$. 
\end{definition}

Now suppose $(\mathcal{C},\varepsilon )$ has kinks
and let $\mathfrak{m}$ be a set of models of $\mathcal{C}$. 
Then obviously $\mathfrak{m}_q:=\mathfrak{m}\cup q\mathfrak{m}$ is a set of models 
of the category $\mathcal{C}_q$. 
Let $i\in \mathbb{Z}$. 
For $x\in \mathfrak{ob}$ we can find
 $u\in \textrm{hom}(x,m)$. By possibly iterated composition with kink morphisms
we can arrange that $\varepsilon (u)\in \{-i,-i-1\}$. 
Thus $q^iu\in \textrm{hom}(q^ix,m)$ or $q^iu\in \textrm{hom}(q^ix,qm)$.
Now assume that $I\subset R$ is an ideal and $\mathfrak{c}$ 
is a model complexity on $\mathfrak{m}$. Then we define $\mathfrak{c}_q$ 
by $\mathfrak{c}_q(m)=\mathfrak{c}_q(qm)=\mathfrak{c}(m)$. Let
$I_q\subset \mathcal{R}[q^{\pm 1}]$ be the ideal generated by $I$ in the
extension. Then it is easy to see that if $\mathfrak{a}$ is linear with respect to $I$ and $\mathfrak{c}$ then $\mathfrak{a}_q$ is linear with respect to
$I_q$ and $\mathfrak{c}_q$.    

Thus we can apply Theorem 2.4 to $(\mathcal{C}_q,\mathfrak{a}_q,\mathfrak{m}_q)$.
and deduce the map
$$\mathfrak{ob}_q\rightarrow (\mathcal{R}[q^{\pm 1}]\mathfrak{m}_q)[[I_q]]/\mathcal{I}_q,$$
with the submodule $\mathcal{I}_q$ defined in the obvious way.

In the topological case, which we consider here,
the passage from $q=1$ to Jones theory is more easily described by 
introducing a \textit{local system} on the space $\textrm{imm}$. This viewpoint has already
been used in \cite{K2}. We refer to there and to \cite{W} for some of the technical details.
This local system appear naturally in trying to work the kink
contributions into the theory as will be seen in the following.

\vskip .1in

Throughout this section we assume from now on $k=0$ and omit the grading index.

\vskip .1in

Let $R:=\mathbb{Z}[q^{\pm 1}]$. 
The idea is that a \textit{universal} Jones type relation
should be of the form:
$$q^{-1}K_+-qK_--K_*$$
respectively 
$$q^{2\epsilon}K_{\epsilon}-K_{-\epsilon}-\epsilon K_*=0$$
for $\epsilon =\pm 1$.

\vskip .1in

Consider the trivial bundle 
over $\textrm{imm}$ with fiber $R$.
Let $\gamma $ be a loop in $\textrm{imm}$ which is tansversal (see section 3). Then the
\textit{oriented intersection number with the discriminant} 
$\delta ' (\gamma )$ is well-defined.
Formally it can be defined as the homomorphism:
$$
\begin{CD}
\delta ': H_1(\textrm{imm})@>\mu >> \mathbb{Z}\mathcal{L}[1]/\mathcal{D}
@>>>\mathbb{Z}, 
\end{CD}
$$
where the second homomorphism is defined my mapping each element of 
$\mathbb{Z}\mathcal{L}[1]$ to the generator $1\in \mathbb{Z}$.
For fixed basepoints $f_{\alpha }$ in the components of $\textrm{imm}$ 
(possibly represented by $\mathfrak{s}(\mathfrak{b})$) let 
$$\pi_1(\textrm{imm}):=\bigcup_{\alpha \in \mathfrak{b}}\pi_1(\textrm{imm};f_{\alpha })$$   

Then define thefollowing  map, which restricts to homomorphisms on the fundamental
groups of the path components of $\textrm{imm}$:
$$\delta : \pi_1(\textrm{imm})\rightarrow \textrm{Aut}(\mathbb{Z}[q^{\pm 1}])$$
is defined by mapping a loop $\gamma $ to the multiplication by
$q^{-2(\delta ' \circ H)(\gamma ))}$. Here we use the Hurewicz homomorphism
$$H: \pi_1(\textrm{imm})\rightarrow H_1(\textrm{imm}).$$ 
The map $\delta $ induces the local system $\widetilde{R}$ on
$\textrm{imm}$. It is called the \textit{Jones local system} since it is
related to the Jones respectively Homfly polynomial skein relation.
 
\begin{remark} 
It is possible to describe the local system \textit{directly} as
a functor on the \textit{transversal fundamental groupoid} of $\textrm{imm}$.
The objects of this category are points in $\textrm{emb}$ and the morphisms
are homotopy classes of paths in $\textrm{imm}$ joining two embeddings. 
Now represent a path in $\textrm{imm}$ with endpoints in $\textrm{emb}$ relative to the 
boundary by a transversal path. Then compute the oriented 
intersection number with the discriminant.
So in this case we apply the homomorphism
$$H_1(\textrm{imm},\textrm{emb})\rightarrow
\mathbb{Z}\mathcal{L}[1]/\mathcal{D}\rightarrow \mathbb{Z}$$
with the first homomorphism defined in the proof of proposition 1.
Now the boundary operator on chains with coefficients in $R$ 
is easily defined as usual using the \textit{parallel transport}
on suitable paths. Thus our construction actually given by a
transversal chain construction as usual in string topology.
\end{remark}

Note that the local system is trivial over $\textrm{emb}$. 
Also by using the kink loops described above we see that
$$H_0(\textrm{imm};\widetilde{R})\cong SR\hat{\pi }/(q^2-1)\cong
S\mathbb{Z}\hat{\pi}\oplus qS\mathbb{Z}\hat{\pi}.$$
Let $\widetilde{\mathcal{D}}$ be the submodule of $R\mathcal{L}[1]$
generated by
all elements of the form
$$q^{-1}K_{*+}-qK_{*-}-q^{-1}K_{+*}+qK_{-*}$$
and let
$$\widetilde{\partial }: R\mathcal{L}[1]/\widetilde{\mathcal{D}}\rightarrow R\mathcal{L}$$
be defined by
$$\widetilde{\partial }(K_*)=q^{-2}K_+-K_-.$$
Note that
$$\widetilde{\partial }(q^{-1}K_{*+}-qK_{*-}-q^{-1}K_{+*}+qK_{-*})=0.$$
The following theorem is proved similarly to Theorem 4.1.

\begin{theorem}
The homology exact sequence for the local system $\widetilde{R}$ on
the pair of spaces $(\rm{imm},\rm{emb})$ is isomorphic to the exact sequence:
$$
\begin{CD}
H_1({\rm emb};R)\rightarrow H_1({\rm imm};\widetilde{R})@>\widetilde{\mu }>>
R\mathcal{L}[1]/\widetilde{\mathcal{D}}@>\widetilde{\partial}>>
R\mathcal{L}@>\widetilde{\mathfrak{h}}>>SR\hat{\pi }/(q^2-1)
\end{CD}
$$
with $\widetilde{\mathfrak{h}}$ onto and the homomorphism
$\widetilde{\mu }$ described below.
\end{theorem}

The homomorphism $\tilde{\mu }$ has the following description:
Consider a $1$-chain representing a homology 
class in $H_1(\textrm{imm};\widetilde{R})$. We can approximate corresponding maps
on $1$-simplices by transversal maps, i.\ e.\ with the boundaries mapping
into $\textrm{emb}$. Now let a simplex $s$ be given by a transversal map 
$$\beta : ([0,1], \{0,1 \}) \rightarrow (\textrm{imm},\textrm{emb})$$
and let $0<t_1<t_2< \ldots t_n<1$ be the parameters which map
into $\textrm{imm}[1]$. Let the sign at $t_i$ be $\epsilon_i \in 
\{0,1 \}$, with $\epsilon_i=-1$ if the path crosses the
discriminant from the \textit{positive to the negative side}.

$$\tilde{\mu }(s)=\sum_{i=1}^n \epsilon_iq^{2(\epsilon_1+\ldots
\epsilon_{i-1})+\epsilon_i}K_{*,i},$$
where $K_{*,i}$ is the isotopy class of $\beta (t_i)$.

\begin{remark} 
Combining $1$-simplices with cancelling boundary terms in a 
$1$-chain it follows easily that there is an epimorphism
$$R\pi_1^0(\textrm{imm})\rightarrow H_1(\textrm{imm};\widetilde{R}),$$
where 
$\pi_1^0(\textrm{imm})$ is the union (over $\mathfrak{b}$) of the 
subgroups of homotopy classes of
those loops with trivial index (thus the $1$-simplex representing 
the loop has trivial boundary with local coefficients).
Note that the effect of changing the basepoint of a loop in $f\in \textrm{emb}$
representing an element in $\pi_1^0(\textrm{imm};f)$ to a basepoint $g$ by pre-
and postcomposition with a path from $g$ to $f$ and the reverse path
multiplies by a power of $q$. Thus we can restrict to a single
basepoint in each component. 
\end{remark}

From the exact sequence in theorem  
we conclude:
\begin{corollary}
There is the short exact sequence
of $R$-modules: 
$$
\begin{CD}
0\rightarrow \left( R\mathcal{L}[1]/\widetilde{\mathcal{D}} \right)
/{\rm im}(\widetilde{\mu }) \rightarrow
R\mathcal{L} \rightarrow 
SR\hat{\pi }/(q^2-1)\rightarrow 0,
\end{CD}
$$
\end{corollary}
This sequence is \textit{never} split in the category of $R$-modules 
(but in the category
of abelian groups). 
Note that, for a given $K\in \mathcal{L}$, there is a unique $\alpha \in
\mathfrak{b}$ such that for each integer number $n$
$$K-q^{2n}K_{\alpha }\in \textrm{ker}(\widetilde{\mathfrak{h}}).$$
In physics terms this could probably be interpreted as follows: 
\textit{For each classical
state we can choose a quantum state up to a certain phase.}

\vskip .1in

By application of the homotopy Conway map (defined in section 4)
the description of the homomorphism $\tilde{\mu }$ given above
precisely corresponds to the map on paths in \cite{K1}.

\begin{proposition}
Suppose that $M$ is aspherical and atoroidal.
Then 
$${\rm im}(\widetilde{ \mu })\subset
\left(R\mathcal{K}/
(\mathcal{D}\cap R\mathcal{K}) \right) \cap
{\rm ker}(\widetilde{\mathfrak{h}})$$
More precisely, each element in the image
of $\tilde{\mu }$ is represented by a sum of elements of the form
$$q^{n_i}(K_{*,i}-K'_{*,i}),$$
where for fixed $i$ all $K_{*,i},K'_{*,i}$ are immersions in
$\mathcal{K}$ resulting from the same link $K\in
\mathcal{L}$ by introducing a kink in one of its components.
\end{proposition}

\begin{proof}
First define homomorphisms for each $\alpha \in \mathfrak{b}$
$$\widetilde{\mu}_{\alpha }: \pi_1(\textrm{imm};f_{\alpha })\rightarrow R\mathcal{L}[1]/\widetilde{\mathcal{D}}$$
using the explicit formula given above. Note that the resulting
homomorphism does not factor through a map defined on the set of
free homotopy classes
of loops in $\textrm{imm}$. Now assume we have given a loop $\rho $ in
$\textrm{imm} $ with index $\varepsilon(\rho )=0$. It is easy to see that there is 
a homotopy to a product
$$\prod_{1\leq i\leq n}\ell _i \rho_i \ell _i^{-1}.$$ 
Here $\ell _i$ is a transverse arc joining $\rho (*)\in \textrm{emb}$
(representing a link $K\in \mathcal{L}$) to the 
basepoint $*_i$ on
$\rho_i$, and each $\rho_i$ is a kink loop (representing $\gamma (K_*)$ for some $K_*$ in $\mathcal{L}[1]$). 
Now each path $\ell_i \rho_i \ell_i^{-1}$ is homotopic inside $\textrm{imm}$ 
to a kink path for the basepoint embedding. Note that the order of
crossing changes can be altered up to $\widetilde{\mathcal{D}}$. This
allows to inductively change each of the paths. 
Thus finally we have a composition of kink loops for the embedding 
$*$ with vanishing 
index because the index is homotopy invariant. This proves the claim. 
Note that the kinks
can still be contained in different components of $K$.
\end{proof}

\section{The framing local systems and its homology sequences}

In this section we use a local system of coefficients
$R'$
on $\textrm{imm}[k]$ naturally related to the study of 
framed links in $M$. The geometric
interpretation of the corresponding homology exact sequence is more delicate than the discussion in the previous section. Note that $\textrm{emb}[0]$ is just the usual space of embeddings 
of oriented links in $M$ with $\mathcal{L}[0]$ the set of isotopy classes.

\vskip .1in

The following extends the set-up discussed in \cite{K3}.
A \textit{total framing} of $f\in \widetilde{\textrm{imm}[0]}$ is 
a choice of equivalence class $[v]$ of a normal vectorfield of $f$
(section of the normal bundle), where we will have $v_1\equiv v_2$ if $v_1$ is
homotopic to $v_2$, or if $v_1$ differs from $v_2$ by twisting the
framings of components of $f$ such that the total number of those
twists add up to zero. We use the natural projections
$$\widetilde{\textrm{imm}}(j)\times C_{2k}(\cup_jS^1)\rightarrow \widetilde{\textrm{imm}}(j)$$
to define the total framing of $f\in \widetilde{\textrm{imm}[k]}$ for $k\geq 1$.
Obviously the set of total framings is in $1-1$
correspondence with the integer numbers. 
The total framings are defined for unordered immersions
by dividing by the actions of symmetric groups in the obvious way.
Now for each $f\in \textrm{imm}[k]$ let $R'_f$
denote the free abelian group gnerated by all total framings on
$f$. Then $R'_f\cong \mathbb{Z}[q^{\pm 1}]=R$ with $q^i$ and $q^j$
indicating two framings of $f$, which differ by $j-i$ twists.

\vskip .1in

Then a local system $R_{\mathfrak{f}}$ on 
$(\textrm{imm},\textrm{emb})$ with bundle of groups given by 
$$\bigcup_{f\in \textrm{imm}}R_f$$
is defined from the collection of homomorphisms, for $f\in \textrm{emb}$:
$$\delta_{\mathfrak{f}}: \pi_1(\textrm{imm};f)\rightarrow \textrm{Aut}(R_f)$$
by assigning to a loop $\gamma $ in $f$ the multiplication by
$q^{\delta '_{\mathfrak{f}}(\gamma )}$. Here 
$\delta '_{\mathfrak{f}}(\gamma )\in
\mathbb{Z}$ is the sum of $2(\varepsilon\circ H)(\gamma )$ and
the framing change
induced by the loop. It is shown in \cite{K3}, using Lin transversality, 
that the homomorphism is well-defined and thus represents a local
system $R_{\mathfrak{f}}$ on $\textrm{imm}[0]$. Moreover, the homology module
$H_0(\textrm{imm}[0];R_{\mathfrak{f}})$ is isomorphic to the skein module
$\mathcal{S}_{\mathfrak{f}}$
defined by dividing $R\mathcal{L}_{\mathfrak{f}}[0]$ by the 
submodule generated by
all elements $q^{-1}K_+-qK_-$ and $q^{-1}K^{(+)}-K$ for $K,K_{\pm }\in \mathcal{L}_{\mathfrak{t}}[0]$. 
Here $\mathcal{L}_{\mathfrak{f}}[0]$ is the usual set of
isotopy classes of framed oriented links and for each framed oriented 
link $K$ the framed link $K^{(+)}$ is the framed link 
defined from $K$ by introducing a 
positive twist into the framing of any of its components.    
The correspondingly defined set of isotopy classes of 
totally framed singular links, equipped with $k$ self-intersections, in $M$ will be
denoted $\mathcal{L}_{\mathfrak{t}}[k]$ and similarly we have the singular framed links $\mathcal{L}_{\mathfrak{f}}[k]$. 

\begin{remark}
The local system above admits 
an interpretation in terms of   
the following general construction: Let $\tilde
X\rightarrow X$ be a (not necessarily regular) covering space over
a connected space $X$. Let $F$ be the fiber over the basepoint $*\in
X$ and let $\pi_1(X;*)\rightarrow \textrm{homeo}(F)$ be the associated
monodromy map. Then there is the natural map
$$\textrm{homeo}(F)\rightarrow \textrm{Aut}(\mathbb{Z}F),$$
where $\textrm{homeo}(F)$ is the group of homeomorphisms of $F$. 
So we can form the associated covering space with fiber $\mathbb{Z}F$
and define the \textit{induced} monodromy by composition in the
obvious way. Now assume that $F$ is a group acting on itself 
by translation. Thus we have a natural map $F\rightarrow \textrm{homeo}(F)$.
Assume that, using a stratification of $X$,
there is defined a second monodromy map:
$$\pi_1(X;*)\rightarrow F\rightarrow \textrm{homeo}(F).$$
The two monodromies into $\textrm{homeo}(F)$ 
can be multiplied, and then mapped to a 
homomorphism into $\textrm{Aut}(\mathbb{Z}F)$
as above.
The fiber of the induced covering is the abelian group $\mathbb{Z}F$
and the monodromy defines a local system, see \cite{W}.
\end{remark}

\begin{lemma} There are the following isomorphisms
$$H_0({\rm imm};R_{\mathfrak{f}})\cong R\mathcal{L}_{\mathfrak{t}}/(q^{-1}K^{(+)}-K)\cong
R\mathcal{L}_{\mathfrak{f}}/(q^{-1}K^{(+)}-K)\cong
\mathbb{Z}\mathcal{L}_{\mathfrak{t}}.$$
The right hand isomorphism is an isomorphism of abelian groups.
\end{lemma}

\begin{proof}
The fiber $R_f$ over 
a given $k$-embedding $f$ is the free abelian
group generated by total framings of $f$.
$R_f$ is identified with $R$ by a choice of total framing of $f$.
Multiplication by $q$ corresponds to a positive twist.
Note that 
$$H_0(\textrm{emb};R_{\mathfrak{f}})\cong \bigoplus _{K\in 
\mathcal{L}}H_0(\textrm{emb}_{K},R_{\mathfrak{f}}),$$
where $\textrm{emb}_K$ is the set of embeddings in the isotopy class $K\in \mathcal{L}_{\mathfrak{t}}$.
Recall that by definition $\textrm{emb}_K$ is the $\textrm{Diff}(M,\textrm{id})$-orbit of a representative
$k$-embedding $f$. Here
$\textrm{Diff}(M,\textrm{id})$ is the group of
diffeomorphisms of $M$, which are isotopic to the identity. 
Then each element of $H_0(\textrm{emb}_K;R_{\mathfrak{f}})$  is represented by an 
element in the fiber over some representative embedding of the isotopy 
class.    
Now the result follows from the definitions of 
$\delta_{\mathfrak{f}}$
and $\mathcal{L}_{\mathfrak{t}}$ and the monodromy interpretation of
$0$-dimensional homology for local systems \cite{W}. 
\end{proof}

We use the notation 
$\mathcal{F}[k]:=R\mathcal{L}_{\mathfrak{t}}[k]/(q^{-1}K^{(+)}-K)$ for $k\geq 0$, and as usual
$\mathcal{F}:=\cup_k\mathcal{F}[k]$.
  
\begin{theorem}
The homology exact sequence for the local system $R_{\mathfrak{f}}$ on the pair
$({\rm imm},{\rm emb})$ is isomorphic with the exact sequence:
$$
\begin{CD}
H_1({\rm emb};R_{\mathfrak{f}})\rightarrow H_1({\rm imm};R_{\mathfrak{f}})@>\mu _{\mathfrak{f}}>>
\mathcal{F}[+1]/\mathcal{D}_{\mathfrak{t}} 
\end{CD}
$$
$$
\begin{CD} 
@>\partial _{\mathfrak{f}}>> 
\mathcal{F}
@>\mathfrak{h}_{\mathfrak{f}}>> \mathcal{S}_{\mathfrak{f}}\rightarrow 0
\end{CD} 
$$
\vskip .1in
\noindent The submodule $\mathcal{D}_{\mathfrak{t}}$, 
the homomorphisms $\mu_{\mathfrak{f}}$, 
$\partial_{\mathfrak{f}}$ and $\mathfrak{h}_{\mathfrak{f}}$ are defined as before
by replacing $\mathcal{L}$ with $\mathcal{L}_{\mathfrak{t}}$. In the case of  
$\mu_{\mathfrak{f}}$ one has to consider the transportation of framings.
\end{theorem}

\begin{proof}
Using Lin transversality and transportation of framings as described above we get the exact sequence
$$
\begin{CD}
H_1({\rm emb};R_{\mathfrak{f}})\rightarrow H_1({\rm imm};R_{\mathfrak{f}})\rightarrow
H_0({\rm imm}[+1];R_{\mathfrak{f}})/\partial '(H_0({\rm imm}[+2];R_{\mathfrak{f}}))
\end{CD}
$$
$$
\begin{CD} 
\rightarrow 
H_0({\rm emb};R_{\mathfrak{f}}) \rightarrow \mathcal{S}_{\mathfrak{f}}\rightarrow 0
\end{CD} 
$$
and then apply Lemma 8.1.
Note that 
$\mathcal{S}_{\mathfrak{f}}$ is isomorphic to the quotient of
$R\mathcal{L}_{\mathfrak{t}}$
by the submodules generated by elements $q^{-1}K^{(+)}-K$
and $q^{-1}K_+-qK_-$ (see \cite{K2}).  
\end{proof}

\begin{remark}
It is interesting to note that the exact sequence above maps into the
sequence of Theorem 4.1 under the the coefficient homomorphism
$R\rightarrow \mathbb{Z}$, which maps $q$ to $1$.
Then the coefficient system becomes trivial and 
$\mathcal{F}[k]$ maps onto 
$\mathbb{Z}\mathcal{L}[k]$.
In this way Theorem 8.1
can be interpreted as a \textit{framing quantization} of Theorem 4.1.
\end{remark}

\begin{corollary}
There is the short exact sequence of
$R$-modules 
$$
\begin{CD}
0\rightarrow \left( \mathcal{F}[+1]/\mathcal{D}_{\mathfrak{t}} \right)/{\rm im}(\mu
_{\mathfrak{f}} )\rightarrow \mathcal{F}\rightarrow 
\mathcal{S}_{\mathfrak{f}}\rightarrow
0
\end{CD}
$$
\end{corollary}

It is known that the
homomorphism $\mathfrak{h}_{\mathfrak{f}}$  
splits in the category of $R$-modules if and
only if each mapping $S^1\times S^1\rightarrow M$ is homotopic into
$\partial M$.

\begin{theorem} \textup{(i)} There is the isomorphism of $R$-modules
$$\mathcal{F}[0] \cong \bigoplus_{\alpha \in
\mathfrak{b}[0]}R(\mathfrak{h}[0]^{-1}(\alpha ))/(q^{2\varepsilon_0(\alpha
)}-1),$$
where for $k\geq 1$
$\mathfrak{h}[k]:  \mathcal{L}[k] \rightarrow \mathfrak{b}[k]$
is the map taking homotopy classes of components.
The index $\varepsilon_{0}(\alpha )\in \mathbb{N}$ is defined by the
the 
$gcd$ over
all absolute values of total
oriented intersection numbers with $2$-spheres in $M$, see \cite{K3} for details.

\noindent \textup{(ii)} There is the isomorphism of $R$-modules
$$\mathcal{S}_{\mathfrak{f}}\cong \bigoplus_{\alpha \in
\mathfrak{b}}R/(q^{2\varepsilon (\alpha )}-1),$$
where $\varepsilon (\alpha )$ is the $gcd$  
of absolute values of intersection numbers of singular tori (defined 
by sweeping a
component $\alpha_i$ through $M$) with a link realizing $\alpha $.
\end{theorem}

\begin{proof}
This is just Chernov's result \cite{Che}, see also the author's approach   
in \cite{K2}. It easily extends to the case $k\geq 1$.
\end{proof}

\begin{remark}
By defining suitable torus maps from $2$-spheres as in (i) above it can be
proved that  
$\epsilon_0(\alpha )$ is a multiple of $\epsilon (\alpha )$ (see \cite{K2}).
The homomorphism $\mathfrak{h}_{\mathfrak{f}}$
 thus is defined by composition of
$\mathfrak{h}$
with some obvious projection. (Use that 
$x^{nm}-1=(x^n-1)(1+x^n+x^{2n}+\ldots +x^{nm})$.)
\end{remark}

\begin{corollary} \textup{(i)} Suppose that $M$ does not contain any
non-separating $2$-spheres. Then
$\mathcal{F}[0] \cong R\mathcal{L}[0]$.

\noindent \textup{(ii)} Suppose that each mapping from $S^1\times S^1$ into $M$ is homologous into $\partial M$. 
$\mathcal{S}_{\mathfrak{f}}\cong SR\hat{\pi}$.
\hfill{$\square$}
\end{corollary}

\begin{proof}
By choosing total framings we can define a section
$\psi: \mathcal{L}\rightarrow \mathcal{L}_{\mathfrak{t}}$ of the forget map
$\phi :\mathcal{L}_{\mathfrak{t}}\rightarrow \mathcal{L}$. 
The induced homomorphism
$\chi : R\mathcal{L}\rightarrow R\mathcal{L}_{\mathfrak{t}}$ is obviously onto.
It follows from Chernov's \cite{Che} results that the fibers of $\phi $ are in
$1$-correspondence with $\mathbb{Z}$. This implies that $\chi $ is
injective. (If $\chi(q^nK-K)=0$ then $\psi (K)$ and its $n$-fold twist
would be isotopic.)
\end{proof}

\begin{corollary}
Suppose that each map of a torus into $M$ is homologous 
into $\partial M$.
Then the homology exact sequence for the local system $R_{\mathfrak{f}}$ on
$({\rm imm}[0],{\rm emb}[0])$ is isomorphic to the exact sequence:
$$
\begin{CD}
H_1({\rm emb}[0];R) \rightarrow H_1({\rm imm}[0];R )@>\mu_{\mathfrak{f}}>>
R\mathcal{L}[1]/\widetilde{\mathcal{D}} 
\end{CD}
$$
$$
\begin{CD}
@>\partial_{\mathfrak{f}}>>R\mathcal{L}[0]@>\mathfrak{h}_{\mathfrak{f}}>>
SR\hat{\pi }
\rightarrow 0
\end{CD}
$$
\hfill{$\square$}
\end{corollary}

\begin{proof}
The assumption implies that the monodromy homomorphism 
$\delta_{\mathfrak{f}}$ is trivial and therefore the homology for the
local system is the usual homology with coefficients in $R$. 
\end{proof}

\begin{remark} The homomorphism $\mathfrak{h}_{\mathfrak{f}}$
composed with the
projection homomorphism $SR\hat{\pi }\rightarrow SR\hat{\pi }/(q^2-1)$
is the homomorphism $\widetilde{\mathfrak{h}}$ in theorem 7.3.
Note that $\mu_{\mathfrak{f}}$ is defined on
$H_1(\textrm{imm}[0];R)\cong H_1(\textrm{imm}[0])\otimes R$
while $\widetilde{\mu }$ is defined on the homology with local
coefficients defined by $\widetilde{R}$.
\end{remark}

\begin{proposition} Suppose that $M$ is apherical and atoroidal. Then
$\mu _{\mathfrak{f}}=0$ for $k=0$.
\end{proposition}
 
\begin{proof} 
It follows from the assumption that in particular each torus is
homologous into the boundary. Thus Corollary 8.2 applies. 
Since $H_1(\textrm{imm}[0];R)\cong H_1(\textrm{imm}[0])\otimes R$ 
it suffices to show that $\mu_{\mathfrak{f}}$ vanishes on the image of the
natural homomorphism of abelian groups
$H_1(\textrm{imm}[0])\rightarrow H_1(\textrm{imm}[0])\otimes R$.
That $\mu_{\mathfrak{f}}$ vanishes on a kink isotopy holds
by definition of the local system.
(In fact the kink isotopy induces a change of framing cancelling
the contribution of the crossing.)
Thus we actually consider a homomorphism
$$H_1(\textrm{map}0])\rightarrow \mathcal{F}[1]/\mathcal{D}_{\mathfrak{f}}
\cong R\mathcal{L}[1]/\widetilde{\mathcal{D}}.$$
The rest of the argument follows word by word the geometric reasoning
in the proof of Theorem 6.1. 
\end{proof}

\section{Proofs of the results on skein modules}

In this section we prove Theorems 1.2, 1.3. and 1.4. concerning the skein module $\mathfrak{J}(M)$ of oriented links in $M$, considered as module over the skein algebra of the ball $\mathfrak{J}(D^3)$, which is isomorphic to $\mathfrak{R}=\mathbb{Z}[q^{\pm 1},z,h,\frac{q^{-1}-q}{h}]$ (\cite{P3} and \cite{Tu}). The framed case $\mathfrak{H}(M)$ is similar and is proved by using the results of section 8, and will not be discussed here in detail.

Recall that $\mathfrak{b}$ is the set of monomials in $\hat{\pi}$ and $\mathfrak{b}_0$ is the set of monomials 
in $\hat{\pi}_0$, where $\hat{\pi}$ and $\hat{\pi}_0$ respectively is the set of free homotopy classes and non-trivial free homotopy classes of oriented loops in $M$ respectively. Each of these sets of monomials will contain the trivial monomial $1$ corresponding to a vaccuum free homotopy class. Then for the corresponding symmetric algebras we have module equalities $S\mathcal{R}\hat{\pi}=\mathcal{R}\mathfrak{b}$ (and $S\mathfrak{R}\hat{\pi}_0=S\mathfrak{R}\hat{\pi}_0$). 
Throughout let $u\in \hat{\pi }$ denote the trivial free homotopy class. We can identify 
$\mathcal{R}\mathfrak{b}=(\mathcal{R}[u])\mathfrak{b}_0$. Note that there is the projection:
$$\mathcal{R}[u]\rightarrow \mathcal{R}[u]/(hu-(q^{-1}-q))\cong \mathfrak{R}:=\mathbb{Z}[q^{\pm 1},z,h,\frac{q^{-1}-q}{h}]\subset \mathbb{Z}[q^{\pm 1},z,h^{\pm 1}].$$
The inclusion is induced by the epimorphism of $\mathbb{Z}[q^{\pm 1},z,h]$-modules: 
$$\mathcal{R}[u]\rightarrow \mathbb{Z}[q^{\pm 1},z,h,\frac{q^{-1}-q}{h}]$$
defined by $u\mapsto \frac{q^{-1}-q}{h}$. Note that each element in $\mathcal{R}[u]$ is \textit{uniquely} represented,
modulo the ideal generated by $hu-(q^{-1}-q)$, in the form $b_0+\sum b_iu^i$ with $b_0\in \mathbb{Z}[q^{\pm 1},z,h]$ and $b_i\in \mathbb{Z}[q^{\pm 1},z]$. Note that $\mathfrak{R}$ is the Rees algebra over the ring 
$\mathbb{Z}[q^{\pm 1},z]$ for the ideal generated by $q^{-1}-q$, see definition in \cite{E}, 6.5.

The proofs will be carried out in two steps. We will first prove the corresponding results for the module $\mathcal{J}(M)$, which is defined without vacuum relations and over the ring $\mathcal{R}:=\mathbb{Z}[q^{\pm 1},z,h]$. 
We have skein relations $q^{-1}K_+-qK_-=\sigma (K_*)$ with $\sigma (K_*)=hK_0$ for a self-crossing and $\sigma (K_*)=ZK_0$ for a crossing of distinct components.

Note also that there is the ring homomomorphism defined by the composition of the inclusion $\mathfrak{R}\subset \mathbb{Z}[q^{\pm 1},z,h^{\pm 1}]$ with the homorphism into the ring $\mathbb{Z}[q^{\pm 1},h^{\pm 1}]$ mapping $z$ to $h$. It is over this ring for which standard Homflypt skein theory is usually discussed. Note that it is a priori not clear that the \textit{model} homomorphism into the corresponding skein module:
$$\mathbb{Z}[q^{\pm 1},h^{\pm 1}]\mathfrak{b}_0\rightarrow \mathcal{S}(M;\mathbb{Z}[q^{\pm 1},h^{\pm 1}])$$ 
with skein relations $q^{-1}K_+-qK_-=hK_0$ for all crossings, including the vaccuum relation. 

We will apply Theorem 1.1 with geometric models corresponding to a homomorphism
$$\mathfrak{s}: \mathcal{R}\mathfrak{b}\rightarrow \mathcal{J}(M),$$
which is injective and thereby the gives a minimal set of models in a corresponding category.
This follows by mapping $q\mapsto 1$ and $z,h\mapsto 0$.
We will then relate $\mathcal{J}(M)$ to $\mathfrak{J}(M)$. Note that there is there natural $\mathcal{R}$-homomorphism $\mathcal{J}(M)\rightarrow \mathfrak{J}(M)$ defined by the identity on $\mathcal{L}(M)$.
The main point is that this homomorphism is compatible with the $\mathcal{R}$-homomorphism $\mathcal{R}[u]\rightarrow \mathfrak{R}$.

Furthermore we would like to point out that our theorems together with Przytycki's universal coefficient theorem \cite{P2} will imply the corresponding results in usual Homflypt skein theory, in particular module isomorphism with $S\mathbb{Z}[q^{\pm 1},h^{\pm 1}]\hat{\pi}_0$.

Note that we have the free $\mathcal{R}$-module $\mathcal{R}\mathfrak{b}\cong S\mathcal{R}\hat{\pi }$, where $\mathfrak{b}$ is the set of monomials in free homotopy classes of loops $\hat{\pi }$ in the fundamental group $\pi_1(M)$ and $S$ denotes the symmetric algebra. 
   
\vskip 0.1in

\begin{proof} \textbf{[of Theorem 1.2]}. We apply Theorem 1.1 in the situation above for the boundary operator
and linear skein potential as defined above. The potential is insensitive, see e.\ g.\ \cite{K1} or \cite{Ka1}. Thus 2.12 applies. It follows from 6.5, and 7.11, which shows that \textit{kink relations} are of the form $(q^{-1}-q)K=h(\mathcal{U}\amalg K)$, where $\amalg$ is disjoint union in a $3$-ball and $\mathcal{U}$ is the unknot, that $U$ is generated by only kink relations. Now we consider the 
commutative diagram with $I$ the ideal generated by $h$:
$$
\begin{CD}
\mathcal{R}\mathfrak{b} @>>> \mathcal{J}(M) @>>> \mathcal{R}\mathfrak{b}[[I]]/U\\
@VVV @VVV @VVV \\
\mathfrak{R}\mathfrak{b}_0 @>\mathfrak{s}>> \mathfrak{J}(M) @>>> \mathfrak{R}\mathfrak{b}_0[[I]]
\end{CD}
$$
It follows from the diagram and the definitions that the composition of the bottom two maps is the inclusion
$\mathfrak{R}\mathfrak{b}_0\rightarrow \mathfrak{R}\mathfrak{b}_0[[I]].$. Now $\mathfrak{R}$ is noetherian as a subring of a noetherian ring. It follows that the inclusion of a free module into the completion is injective, see the corollary of Krull's theorem (\cite{AM}, p.\ 110). Therefore the image of $\mathfrak{s}$ in the skein module $\mathfrak{J}(M)$ is isomorphic to $\mathfrak{R}\mathfrak{b}_0$.
\end{proof}

\begin{remark} A proof of Theorem 1.2 can also be given by extending the $2$-category $\mathcal{C}_q$ from Theorem 7.1 to a category $\mathcal{C}_q^u$ so that objects are of the form $u^iq^jK$ for $i$ nonnegative integers and $j\in \mathbb{Z}$, with $K$ having no trivial components separated in a $3$-ball. 
Similarly $1$ and $2$-morphisms can be extended (see the argument below concerning the action of $i$ on 
the set of singular links).
Let $\mathcal{R}_0:=\mathbb{Z}[z,h]$ so that $\mathcal{R}=\mathcal{R}_0[q^{\pm 1}]$. Let $\mathcal{L}_0(M)$ denote the set of isotopy classes of oriented links in $M$ without any trivial components in separated $3$-balls. Note that isotopy of links in $M$ is in one-to-one correspondence with isotopy in $M\setminus D$ for an open $3$-ball $D$. Then, after 
fixing a $3$-ball near the boundary there is a well-defined action of a non-negative integer $i$ on spaces of singular links in $M\setminus D$ defined by taking a disjoint union with an $i$-component unlink in that ball. This actions descends to isotopy classes. In particular we can identify, denoting $M\setminus D$ by just $M$, $\mathcal{L}(M)=\cup_i u^i\mathcal{L}_0(M)$.
Now note that there isomorphisms of $\mathcal{R}_0$-modules $\mathcal{R}_0\textrm{ob}(\mathcal{C}_q^u)\cong \mathcal{R}[u]\mathcal{L}_0(M)\cong \mathcal{R}\mathcal{L}(M)$. Now for each $K\in \mathcal{L}(M)$ let $\chi (K)\geq 0$ be the number of components of a maximal unlink separated from the remaining component of $K$. Then the $\mathcal{R}$-module $\mathfrak{J}(M)$
is isomorphic to the quotient of the free $\mathcal{R}[u]$-module with basis $\mathcal{L}_0(M)$ by the submodule generated by $q^{-1}K_+-u^{\chi (K_-)}qK_-=\sigma (K_*)u^{\chi (K_0)}K_0$ with $K_+\in \mathcal{L}_0(M)$, a similar relation for $K_-\in \mathcal{L}_0(M)$, and $q^{-1}-q=hu$. In this way $\mathfrak{J}(M)$ is in fact an $\mathcal{R}[u]$-module. Now the mapping $\mathcal{R}[u]\rightarrow \mathfrak{R}$, $u\mapsto \frac{q^{-1}-q}{h}$ is an $\mathcal{R}[u]$-isomomorphism to the $\mathfrak{R}$-module $\mathfrak{J}(M)$. Note that each $\mathfrak{R}$-module is also a $\mathcal{R}[u]$-module using the mapping above. We can now apply Theorem 2.4 
to the category $\mathcal{C}_q^u$ and Theorem 6.1 and Proposition 7.1 as before. 
\end{remark}

\begin{proof} \textbf{[of Theorem 1.3]}. Using the homomorphism defined in Theorem 1.1 and the Proof of Theorem 1.2 above we get for each $\beta \in \mathfrak{b}_0$ the homomorphism
$$\iota_{\beta }: \mathfrak{J}(M)\rightarrow \mathfrak{R}\mathfrak{b}_0[[h]]\rightarrow \mathfrak{R}[[h]]$$
with the right hand map induced by the projection $\mathfrak{R}\mathfrak{b}_0\rightarrow \mathfrak{R}$ defined by $\beta $. 
\end{proof}

\begin{remark} Recall that there is a ring homomorphism $\mathfrak{R}\rightarrow \mathbb{Z}[q^{\pm 1},h^{\pm 1}]$, which then can be mapped in various ways to the power series ring $\mathbb{Q}[[x]]$. These ring homomorphisms extend to
$$\mathfrak{R}[[h]]\rightarrow \mathbb{Q}[[x]]$$
and we recover the power series invariants of links in $3$-manifolds as defined in \cite{KL}, \cite{Ka2}. 
\end{remark}

\begin{proof} \textbf{[of Theorem 1.4]}. Lens spaces $M=L(p,q)$ are aspherical and atoroidal for $p\neq 0$. Thus Theorems 1.2 and 1.3 apply. A careful checking of the reduction arguments in (\cite{C}, Proof of Theorem 1.1), or in \cite{GM}), shows that the homomorphism $\mathfrak{R}\mathfrak{b}_0\rightarrow \mathfrak{J}(M)$ is onto. 
\end{proof}

\begin{remark} (i) The argument in \cite{Tu}, 5.3 shows that the skein module over $\mathbb{Z}[q^{\pm 1},h^{\pm 1}]$, which is not distinguishing between self-crossings and mixed crossings, does not contain more information about links than the module $\mathcal{L}(M)$.

\noindent (ii) The study of Homflypt modules for Lens spaces $L(p,1)$ has a long history, in particular the question of the minimal generating set has been settled a long time ago, see \cite{DLP}, \cite{DL1} and \cite{DL2}. The result for $L(p,1)$ has been first proved in \cite{GM}. 

\end{remark}

\section{Relations with string topology} 

In order to relate the homomorphisms
$\mu $ respectively $\mu_0$ with string topology operations
we have to deal with both a passage from
isotopy to homotopy, and a multiplication (respectively 
transversely a smoothing)
operation. In turns out to be
interesting to describe this in the two possible ways of applying
these operations in different order. Recall that string topology operations in our 
case are defined on the chain level using bordism instead of homology, i.\ e.\ 
representing homology classes by mappings oriented manifolds for which
the usual transversality arguments hold.  

The ad hoc arguments used in the proofs of our main results 
hint at difficulties in a
passage from ordered to unordered maps.
The Chas Sullivan construction is in fact a construction in the
homology of \textit{ordered} maps of circles into $M$.

We first recall basic features of their set-up and restrict our
viewpoint at this moment to the two fundamental string topology
operations in the case of $3$-manifolds.
Let $\widetilde{\textrm{top}}(j)$ denote the space of continuous (or piecewise
smooth) mappings $\cup_jS^1\rightarrow M$. Moreover let 
$\widetilde{\textrm{top}}(j)_o$ denote the subspace of those maps with at least one
constant component.
The group $(S^1)^j$ acts on the space $\widetilde{\textrm{top}}(j)$ preserving
the subspace $\widetilde{\textrm{top}}(j)_o$. This can be used to define
equivariant homology groups $H_*^{\textrm{eq}}(\widetilde{\textrm{top}}(j))$
and relative equivariant homology groups 
$H_*^{\textrm{eq}}(\widetilde{\textrm{top}}(j),\widetilde{\textrm{top}}(j)_o)$.

In \cite{CS1} and \cite{CS2} the following two basic string operations 
$c$ and $s$ are defined, which we call the \textit{collison} and \textit{self-collision} operators. 
These are homomorphisms
$$
c: H_1^{\textrm{eq}}(\widetilde{\textrm{top}}(2))\rightarrow H_0(\textrm{top}(1))
$$
and
$$
s: H_1^{\textrm{eq}}(\textrm{top}(1))\rightarrow H_0(\widetilde{\textrm{top}}(2),\widetilde{\textrm{top}}(2)_o)
$$

\begin{remark}
In the case of rational coefficients we have the K\"{u}nneth isomorphisms
$$H_*^{\textrm{eq}}(\widetilde{\textrm{top}}(j),\widetilde{\textrm{top}}(j)_o)\cong 
\bigotimes_jH_*^{\textrm{eq}}(\widetilde{\textrm{top}}(1),\widetilde{\textrm{top}}(1)_o).$$
In \cite{CS2} the operations $c$ and $s$ are           
described geometrically 
in terms of $H_*^{\textrm{eq}}(\widetilde{\textrm{top}}(j),\widetilde{\textrm{top}}(j)_o)$. 
But Chas and Sullivan prefer to use the K\"{u}nneth
identification to express the operations in terms of
$H_*(\textrm{top}(1),\textrm{top}(1)_0)$
denoted $\mathbb{L}$ in \cite{CS2}.
We have already identified the equivariant and non-equivariant
$0$-dimensional homology groups. 
\end{remark}

To avoid cluttering of notation we will not indicate the grading in spaces of immersions in the rest of the paper so that $\textrm{imm}$ is used denote \textbf{only} $\textrm{imm}[0]$, and similarly for the ordered versions. 
It follows easily from the definitions that operations, also
denoted $s$, $c$ here, can be defined as follows
$$s: H_1^{\textrm{eq}}(\textrm{imm}(1))\rightarrow H_0(\widetilde{\textrm{imm}}(2)).$$ 
and similarly 
$$c: H_1^{\textrm{eq}}(\widetilde{\textrm{imm}}(2))\rightarrow H_0(\textrm{imm}(1))$$
Here we used that the inclusion from continuous maps into
immersions is an isomorphism in $0$ dimensional homology.

The Chas Sullivan construction allows to
introduce \textit{dummy} components, which are insensitive to
collisions or self-collisions. An example would be
$$s_i: H_1^{\textrm{eq}}(\widetilde{\textrm{imm}}(j))\rightarrow
H_0(\widetilde{\textrm{imm}}(j+1)),$$
which measures self-collisions in the $i$-th component, $1\leq i\leq j$. 
In the same way there are defined for $j\geq 2$:
$$c_{k\ell}: H_1^{\textrm{eq}}(\widetilde{\textrm{imm}}(j))\rightarrow H_0(\widetilde{\textrm{imm}}(j-1)),$$
measuring the collisions between the $k$-th and $\ell$-th component for
$1\leq k<\ell \leq j$.
By taking the sum of the $s_i$ respectively $c_{k\ell}$ we get
well-defined homomorphisms with domain and target as above.
Finally by summation over all $j$ we have defined the homomorphisms:
$$s, c: H_1^{\textrm{eq}}(\widetilde{\textrm{imm}})\rightarrow H_0(\widetilde{\textrm{imm}})$$
and thus the homomorphism $(s,t)$. This can be precomposed with
the surjective homomorphism
$$H_1(\widetilde{\textrm{imm}})\rightarrow H_1^{\textrm{eq}}(\widetilde{\textrm{imm}})$$
to define
$$\widetilde{\mathfrak{j}}: H_1(\widetilde{\textrm{imm}})\rightarrow H_0(\textrm{imm})\oplus H_0(\textrm{imm}),$$
where we have in the domain we have have composed with the
mapping induced by the projection $\widetilde{\textrm{imm}}\rightarrow \textrm{imm}$.  
(The surjectivity is proved as follows: First decompose
$$H_1^{\textrm{eq}}(\widetilde{\textrm{imm}})\cong
\bigoplus_{a\in \widetilde{\mathfrak{b}}}H_1^{eq}(\widetilde{\textrm{imm}}_a),$$
where the $\widetilde{\mathfrak{b}}$ is the set of ordered sequences
in $\hat{\pi}$. Then apply the K\"{u}nneth theorem to decompose
for $a$ of length $n$: 
$$H_1^{eq}(\widetilde{\textrm{imm}_a})\cong \bigotimes_{1\leq i\leq
n}H_1(\textrm{imm}_{a_i}/S^1).$$
Consider the Gysin epimorphisms, see \cite{K2} and \cite{CS1}
$$H_1(\textrm{imm}_{a_i})\rightarrow H_1(\textrm{imm}_{a_i}/S^1).$$
Finally apply the K\"{u}nneth isomorphism in non-equivariant
homology to deduce the result.) 

\vskip .1in 

Consider the decomposition
$$H_1(\widetilde{\textrm{imm}})\cong \bigoplus_{a\in \widetilde{\mathfrak{b}}}H_1(\textrm{imm}_a).$$
and the Hurewicz epimorphisms for $a\in \mathfrak{b}$:
$$\pi_1(\widetilde{\textrm{imm}};f_a)\rightarrow H_1(\widetilde{\textrm{imm}}_a).$$ 
The definitions of $c$ and $s$ imply that
$\widetilde{\mathfrak{j}}$ factors through the image of
$\pi_1(\widetilde{\textrm{imm}}_a)$ in $\pi_1(\textrm{imm}_{\alpha })$, where 
$\alpha \in \mathfrak{b}$ is the unordered sequence underlying $a$. 
 We know that this homomorphism
\textit{naturally} extends to a \textit{homomorphism} on 
the full groups $\pi_1(\textrm{imm}_{\alpha })$.
Because commutators map trivially it factors through the
homology groups $H_1(\textrm{imm}_{\alpha })$. Using linearity and
summarizing we thus have
defined the homomorphism:
$$\mathfrak{j}: H_1(\textrm{imm})\rightarrow H_0(\textrm{imm})\oplus H_0(\textrm{imm}).$$

We want to consider $\mathfrak{j}$ with the homomorphism $\mu $. 
We first show that $\mathfrak{j}$ results from $\mathfrak{mu }$ by  
applying a smoothing construction followed by the passage from isotopy
to homotopy.

\vskip .1in

Let
$$\mathfrak{c}: \mathbb{Z} \mathcal{L}[1]\rightarrow
\mathbb{Z}\mathcal{L}\oplus \mathbb{Z}\mathcal{L}$$
be defined by mapping basis elements $K_*$ to the \textit{Conway smoothing}
$K_0$, placed
into
the first summand for a self-crossing and into the second summand for
a crossing of dictinct components of $K_*$.
 
Recall that for $j\geq 0$ let $\mathcal{L}[k](j)$ denotes the subset of
$\mathcal{L}[k]$ given by immersions with $j$ components.
Then more precisely the homomorphism defined above is graded by 
homomorphisms:
$$\mathfrak{c}(j): \mathbb{Z} \mathcal{L}[1](j)\rightarrow \mathbb{Z}\mathcal{L}(j+1)\oplus
\mathbb{Z}\mathcal{L}(j-1)$$
for $j\geq 1$. Of course $\mathfrak{c}$ induces the homomorphism
$$\mathbb{Z}\mathcal{L}[1]/\mathcal{D}\rightarrow \mathbb{Z}\mathcal{L}\oplus
\mathbb{Z}\mathcal{L}/\mathfrak{c}(\mathcal{D}).$$ 

It follows that there is the well-defined homomorphism
$$\mathfrak{c}\circ \mu: H_1(\textrm{imm})\rightarrow 
(\mathbb{Z}\mathcal{L}\oplus
\mathbb{Z}\mathcal{L})/\mathfrak{c}(\mathcal{D}).
$$
assigning to each transversal loop in $\textrm{imm}$ the oriented sum of Conway
smoothings at singular parameters. We call this map the \textit{Lin homomorphism}.

Now note that 
$$\mathfrak{h}\oplus \mathfrak{h}: \mathbb{Z}\mathcal{L}\oplus
\mathbb{Z}\mathcal{L}\rightarrow \mathbb{Z}\mathfrak{b}\oplus
\mathbb{Z}\mathfrak{b}$$ 
maps $\sigma (\mathcal{D})$ to $0$.
Thus we have defined the homomorphism
$$(\mathfrak{h}\oplus \mathfrak{h})\circ \mathfrak{c}\circ \mu :H_1(\textrm{imm})\rightarrow \mathbb{Z}\mathfrak{b}\oplus
\mathbb{Z}\mathfrak{b}\cong
H_0(\textrm{imm})\oplus H_0(\textrm{imm}).$$

The following result is now obvious from the definitions and compares our constructions with those of Chas and Sullivan .

\begin{theorem}
$$(\mathfrak{h}\oplus \mathfrak{h})\circ \mathfrak{c}\circ \mu =
\mathfrak{j}: H_1({\rm imm}) \rightarrow H_0({\rm imm})\oplus H_0({\rm imm})$$
\end{theorem}

\end{document}